\documentclass[a4paper,11pt]{amsart}
\usepackage{amssymb}
\usepackage{amscd}
\usepackage{comment}
\usepackage{amsmath,amsthm}
\usepackage[colorlinks=true]{hyperref}
\usepackage{enumerate}
\usepackage{booktabs,multirow}
\usepackage{tikz}
\usepackage{rotating}
\usetikzlibrary{patterns}
\usetikzlibrary{decorations.pathreplacing}
\usetikzlibrary{calc,through}

\allowdisplaybreaks[1]
\numberwithin{figure}{section}
\numberwithin{equation}{section}

\title{From dimer models to generalized lattice paths}
\author[K.~Shigechi]{Keiichi~Shigechi}
\email{k1.shigechi AT gmail.com}
\date{\today}

\newcommand\tikzpic[2]{
\raisebox{#1\totalheight}{
\begin{tikzpicture}
#2
\end{tikzpicture}
}}
\newtheorem{theorem}[figure]{Theorem}
\newtheorem{example}[figure]{Example}
\newtheorem{lemma}[figure]{Lemma}
\newtheorem{defn}[figure]{Definition}
\newtheorem{prop}[figure]{Proposition}
\newtheorem{cor}[figure]{Corollary}

\newtheorem{remark}[figure]{Remark}
\begin{document}
\begin{abstract}
A recurrence relation of the generating function of the dimer model
of Fibonacci type gives a functional relation for formal power series 
associated to lattice paths such as a Dyck, Motzkin and Schr\"oder path.
In this paper, we generalize the correspondence to the case of 
generalized lattice paths, $k$-Dyck, $k$-Motzkin and $k$-Schr\"oder paths, by modifying 
the recurrence relation of the dimer model.
We introduce five types of generalizations of the dimer model by keeping 
its combinatorial structures.
This allows us to express the generating functions in terms of generalized lattice paths.
The weight given to a generalized lattice path involves 
several statistics such as size, area, peaks and valleys, and 
heights of horizontal steps.
We enumerate the generalized lattice paths by use of the 
recurrence relations and the Lagrange inversion theorem. 
\end{abstract}

\maketitle

\section{Introduction}
The generating functions $G_{n}(r):=G_{n}(m,r,s,u)$ of 
one-dimensional multi-colored dimer models of length $n$ satisfy 
the three-term recurrence relation of Fibonacci type:
\begin{align}
\label{eq:gffib}
G_{n}(r)=f(rs)G_{n-1}(rs)+g(rs)G_{n-2}(rs^2),
\end{align}
where $f(x)$ and $g(x)$ are functions which reflect the combinatorial structures of the models.
In \cite{S21b}, the author studied the relation between the recurrence relation (\ref{eq:gffib})
and combinatorial lattice paths such as Dyck, Motzkin and Schr\"oder paths.
These lattice paths are well-studied combinatorial objects.
For example, the total number of Dyck path of length $n$ is given by the Catalan number 
(see e.g. \cite{Sta97,Sta972,Sta15} and references therein). 
Similarly, Motzkin paths are counted by the Motzkin numbers \cite{Aig98,BirGilWei14,DonSha77,OstderJeu15}, 
and Schr\"oder paths are by the large Schr\"oder numbers \cite{Sch70}.
The generating functions of multi-colored dimer models in one-dimension are expressed 
as formal power series of lattice paths with some weights.
The key observation for this correspondence is to introduce a new formal power series 
$\chi(r)$ defined by 
\begin{align*}
\chi(r):=\lim_{n\rightarrow\infty}\genfrac{}{}{}{}{G_{n-1}(rs)}{G_{n}(r)}.
\end{align*}
The relation (\ref{eq:gffib}) in the large $n$ limit gives the functional relation for $\chi(r)$.
The generating function $G_{n}(r)$ possesses a nice behavior in the large $n$ limit 
such that the identities of Rogers--Ramanujan type \cite{Mac1918,Rog1894,Sch1917} naturally appear.
Even when $n$ is finite, $G_{n}(r)$ is shown to have a similar behavior as $n$ infinity
through the formal power series $\chi(r)$ of the lattice paths with the correction 
coming from the finiteness of the system \cite{S21b}.

In this paper, we generalize the correspondence between the recurrence relation (\ref{eq:gffib}) and the 
combinatorial lattice paths to the case of generalized dimer models and generalized lattice paths.
For this purpose, we first introduce several recurrence relations which arise from 
generalized dimer models explained below. Then, by extending the correspondence mentioned above, 
we study the formal power series associated to generalized lattice paths.
Combinatorial lattice paths appearing in this context are $k$-Dyck paths, 
generalized Motzkin paths, and two types of generalized Schr\"oder paths (type A and type B).
Here, $k$-Dyck paths are well-studied generalized lattice paths, whose total 
number of length $n$ is given by the Fuss--Catalan number \cite{Fus1791}.
The generalized Motzkin and Schr\"oder paths are the paths obtained from 
Motzkin and Schr\"oder paths by a similar generalization as in the 
case of Dyck paths.
 
The correspondences between recurrence relations for dimer models and formal power 
series for  lattice paths  are summarized as below.
We introduce five types of the recurrence relations for the dimer models.
For the dimer model of each type, we have a formal power series of lattice 
paths.
Each formal power series is expressed as a sum of lattice paths with certain weights.
The weights given to a lattice path depends on each type. 
Relevant statistics on a lattice path are size, area, peaks and valleys, heights of horizontal 
steps (if it has), and some patterns in the path. 
The five types are as follows:
\begin{enumerate}
\item 
Type I.
This model is the simplest generalization of the recurrence relation (\ref{eq:gffib}).
The corresponding lattice paths are $k$-Dyck paths, $k$-Schr\"oder paths of type A
and $k$-Motzkin paths.
At the specialization of a variable, the generating function gives 
the numbers of $k$-Dyck paths, i.e., Fuss--Catalan numbers, and generalized large Schr\"oder numbers.
\item
Type II.
In the original dimer model, two connected components of dimers 
have distance at least one site. The model of type II is characterized 
by the following condition: two connected components of dimers 
are separated at least two sites.
The corresponding lattice paths are $k$-Dyck paths.
\item 
Type III.
In this model, the dimer has length more than one. 
The lattice paths are $k$-Dyck paths, and generalized lattice paths 
bijective to $k$-Schr\"oder paths of type $B$.
The weight of a path comes from the statistics area and 
peaks of the path.
A specialization of formal variables gives the cardinality 
of $k$-Schr\"oder paths of type $B$.
\item 
Empty dimers.
The recurrence relation of this type focuses on the ``empty" dimers,
namely, we exchange the roles of dimers and empty sites.
The corresponding lattice paths are $k$-Dyck paths and $k$-Schr\"oder path 
of type B.
The formal power series for $k$-Dyck paths has a pole at $m=1$, 
where the variable $m$ counts the number of colors of dimers. 
The degree of the pole depends on both two parameters $k$ and the system 
size $n$.
The weight of a $k$-Dyck path reflects a certain pattern structure 
in the path.
\item 
The dimer model with distance two.
In this model, dimers of the same color are separated at least two sites.
The generating function for the dimer model satisfies a recurrence relation of degree four.
At the specialization of a variable $m$, the generating function of the lattice 
paths satisfies the recurrence relation for $k$-Dyck paths with simple weights.
For general $m$, to determine the correct weight given to a $k$-Dyck path still remains 
as an open problem (see Section \ref{sec:OP}).
\end{enumerate}
In all five cases, the recurrence relations for the dimer models are generalized 
to preserve their combinatorial structures.
The recurrence relations contain only the terms of the form $G_{n-p}(rs^{p})$ 
with a non-negative integer $p$. 
This allows us to express formal power series, which come from the recurrence 
relations for the dimer models, in terms of $k$-Dyck paths in the large system 
size.
Note that the other two generalized lattice paths, $k$-Motzkin and $k$-Schr\"oder 
paths, contain $k$-Dyck paths as a subset.
Thus, by modifying the weight given to $k$-Dyck paths, the formal power 
series have expressions in terms of $k$-Motzkin and/or $k$-Schr\"oder paths.

Since a formal power series  satisfy a functional relation coming from 
Eq. (\ref{eq:gffib}), we enumerate the generalized lattice paths associated 
to the formal power series by applying the Lagrange inversion theorem.
We give the cardinalities of $k$-Schr\"oder paths of type $A$, that of $k$-Motzkin paths, 
and $k$-Schr\"oder paths of type B respectively in Proposition \ref{prop:kSchA}, 
Proposition \ref{prop:kMot} and Proposition \ref{prop:SchB}.

The paper is organized as follows.
In Section \ref{sec:LP}, we introduce generalized Dyck, Motzkin and 
Schr\"oder paths. To enumerate these combinatorial paths, we also 
recall the Lagrange inversion formula.
In Section \ref{sec:Dimer}, we generalize the recurrence relations (\ref{eq:gffib})
to those of several dimer models.
Then, in Section \ref{sec:gfGLP}, we study the generating functions 
of the generalized lattice paths introduced in Section \ref{sec:LP}, 
which come from the recurrence relations studied in Section \ref{sec:Dimer}. 
Section \ref{sec:OP} is devoted to open problems regarding 
to generating functions of the lattice paths.

\section{Lattice paths}
\label{sec:LP}
In this section, we introduce three types of generalized lattice paths: 
$k$-Dyck, $k$-Motzkin and $k$-Schr\"oder paths.
When $k=1$, these generalized paths coincide with the standard notions 
of Dyck, Motzkin and Schr\"oder paths.

We also recall the Lagrange inversion theorem. This theorem gives the Taylor series of the inverse function 
of an analytic function.

\subsection{Generalized Dyck paths}
\subsubsection{Generalized Dyck paths}
Let $k\ge1$ be a positive integer.
A {\it generalized Dyck path}, or a $k$-Dyck path, of size $n$ is a lattice path 
from $(0,0)$ to $(kn,n)$ such that 
\begin{enumerate}
\item it consists of up steps $(0,1)$ and down steps $(1,0)$, 
\item it never goes below the line $y=x/k$.
\end{enumerate}
When $k=1$, the notion of $k$-Dyck paths coincides with the standard notion 
of Dyck paths.
By replacing an each step in a $k$-Dyck path $\mu$  by $U$ or $D$, we write 
$\mu$ as a word of $U$ and $D$ in one-line notation. 

\begin{defn}
We denote the set of $k$-Dyck paths of size $n$ by $\mathtt{Dyck}(n;k)$.
\end{defn}
For example, $\mathtt{Dyck}(2;2)$ is the set of $2$-Dyck paths:
\begin{align*}
UDDUDD, \quad UDUDDD, \quad UUDDDD.
\end{align*}

It is well-known that the generating function for Dyck paths ($k=1$) is given by
\begin{align}
\label{eq:gfDyck2}
\begin{split}
\sum_{0\le n}|\mathtt{Dyck}(n;1)|x^{n}&=\genfrac{}{}{}{}{1-\sqrt{1-4x}}{2x}, \\
&=1+x+2x^2+5x^3+14x^4+42x^{5}+132x^{6}+\cdots.
\end{split}
\end{align}
This integer sequence appears as A000108 in OEIS \cite{Slo}.
The generating function of $|\mathtt{Dyck}(n;1)|$ can be simply derived by 
writing down the recurrence relation, which is degree two. 
Actually, if we write the left hand side of Eq. (\ref{eq:gfDyck2}) as $y$,
$y$ satisfies $y=1+xy^2$.

Similarly, the number of the $k$-Dyck paths of size $n$ is given by 
the well-known Fuss--Catalan number \cite{Fus1791}:
\begin{align*}
|\mathtt{Dyck}(n;k)|=\genfrac{}{}{}{}{1}{(k+1)n+1}\genfrac{(}{)}{0pt}{}{(k+1)n+1}{n}
=\genfrac{}{}{}{}{1}{kn+1}\genfrac{(}{)}{0pt}{}{(k+1)n}{n}.
\end{align*} 
We denote by $\mu_{0}$ (resp. $\mu_{1}$) the lowest (resp. highest) path 
in $\mathtt{Dyck}(n;k)$.
If we write an up (resp. down) step as $U$ (resp. $D$), 
$\mu_{0}$ and $\mu_{1}$ are expressed as 
\begin{align*}
\mu_{0}:=(UD^{k})^{n}, \qquad \mu_{1}:=U^{n}D^{kn}.
\end{align*}
Since a $k$-Dyck path $\mu$ is above $\mu_{0}$ and below $\mu_{1}$, 
one can consider a region surrounded by $\mu$ and $\mu_{0}$.
Given a $k$-Dyck path $\mu$, we denote by $\mathrm{Y}(\mu_{0}/\mu)$ 
the skew Young diagram of shape $\mu_{0}/\mu$.	
We denote by $|Y(\mu_{0}/\mu)|$ the number of unit boxes above 
$\mu_{0}$ and below $\mu$.

\begin{example}
Let $(n,k)=(4,2)$. We have seven $2$-Dyck paths with $|Y(\mu_{0}/\mu)|=3$.
The seven paths are 
\begin{gather*}
U^{2}D^{3}UD^{3}UD^{2}, \qquad U^{2}D^{4}UDUD^{3}, \qquad UDUDUD^{4}UD^{2}, 
\qquad UDUD^{2}UD^{2}UD^{3}, \\
UDUD^{3}U^{2}D^{4}, \qquad UD^{2}U^{2}D^{3}UD^{3}, \qquad UD^{2}UDUDUD^{4}.
\end{gather*}
The total number $|\mathtt{Dyck}(4;2)|$ of $2$-Dyck paths of size $4$ 
is $55$.
\end{example}

\begin{defn}
A {\it prime} $k$-Dyck path is a $k$-Dyck path such that it never touches 
the line $y=x/k$ except the starting and ending points.
\end{defn}
For example, the $2$-Dyck path $UDDUDD$ is not prime, 
and the $2$-Dyck path $UDUDDD$ is prime.

Note that any $k$-Dyck path $\mu$ can be expressed as a concatenation $\mu=\mu^{p}\circ\mu'$ of
two $k$-Dyck paths $\mu^{p}$ and $\mu'$ where $\mu^{p}$ is a prime $k$-Dyck path.
In this decomposition, $\mu'$ may be an empty path.
One can obtain a functional relation for the generating function of the total number of 
$k$-Dyck paths of size $n$ by use of this decomposition.
Further, by applying Lagrange inversion theorem \ref{thrm:Lag}, we obtain 
the Fuss--Catalan numbers.

\subsubsection{$k+1$-ary trees}
A {$(k+1)$-ary tree} is a rooted tree such that each node has exactly $k+1$ children.
A special cases is $k=1$, and we obtain binary trees.
Similarly, we have ternary trees in the case of $k=2$.
The size of a $(k+1)$-ary is the number of internal nodes (including the root).

We construct a bijection between $k$-Dyck paths size $n$ and $(k+1)$-ary 
trees of size $n$.
Given a $k$-Dyck path $\mu$, we assign a $(k+1)$-ary tree $\mathcal{T}(\mu)$
recursively as follows.

A $k$-Dyck paths of size $0$, which is the empty path, corresponds to 
the empty tree.
For $n\ge1$, a $k$-Dyck path $\mu$ has the following form:
\begin{align}
\label{eq:mudecomp}
\mu=U\circ \mu_1\circ D \circ \mu_2\circ D\circ \cdots \circ \mu_{k} D\circ \mu_{k+1}, 
\end{align}
where $\mu_{i}$, $1\le i\le k+1$, are $k$-Dyck paths.
Recall that a $(k+1)$-ary tree has $(k+1)$ children.
Let $\mathcal{T}_{i}:=\mathcal{T}(\mu_{i})$ be a $(k+1)$-ary tree 
corresponding to the path $\mu_{i}$.
Then, the $(k+1)$-ary tree $\mathcal{T}(\mu)$ is obtained from the $(k+1)$-ary 
tree $T$ of size one by attaching $\mathcal{T}_{i}$, $1\le i\le k+1$, to 
the $i$-th (from left) children of $T$.

In the decomposition (\ref{eq:mudecomp}), we have a prime $k$-Dyck path
if and only if $\mu_{k+1}=\emptyset$.

\begin{remark}
\label{rmk:DyckTree}
Since the decomposition (\ref{eq:mudecomp}) of $\mu$ is unique, the map from a $k$-Dyck 
path to a $(k+1)$-ary tree has an obvious inverse and it is unique.
Therefore, we have a bijection between them.
\end{remark}

An example of ternary tree corresponding to a $2$-Dyck path is shown 
in Figure \ref{fig:extree}.
The corresponding $2$-Dyck path is $UDUD^2UD^3UD^2$, and 
its decomposition of the form (\ref{eq:mudecomp}) is given by $UD(UD^2UD^2)D(UD^2)$.
\begin{figure}[ht]
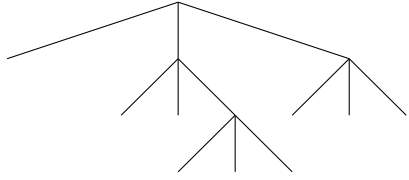

\tikzpic{-0.5}{[scale=0.5]
\coordinate
	child{coordinate}
	child[missing]
		child[missing]
	child{coordinate
		child{coordinate}
		child{coordinate}
		child{coordinate
			child{coordinate(c5)}
			child{coordinate(c6)}
			child{coordinate(c7)}	
		}
	}
	child[missing]
	child[missing]
	child{coordinate(c4)
		child{coordinate(c9)}
		child{coordinate(c10)}
		child{coordinate(c11)}
	};
}
\caption{A ternary tree corresponding to the $2$-Dyck path $UDUD^2UD^3UD^2$}
\label{fig:extree}
\end{figure}
Therefore, the $2$-Dyck path is not prime.

\begin{defn}
We denote by $\mathtt{Tree}(n;k)$ the set of $(k+1)$-ary trees of size $n$.
\end{defn}
From Remark \ref{rmk:DyckTree}, it is clear that 
the number $|\mathtt{Tree}(n;k)|$ is equal to the number of $k$-Dyck paths 
of size $n$, i.e., the Fuss--Catalan number.

\subsection{Generalized Motzkin paths}
A {\it Motzkin path} is a lattice path from $(0,0)$ to $(n,0)$ in 
the first quadrant such that each step is either an up step $U=(1,1)$, 
a down step $D=(1,-1)$ or a horizontal step $(1,0)$.
As in the case of Dyck paths, Motzkin paths never go below the horizontal 
line $y=0$.
For example, we have four Motzkin paths of size $3$:
\begin{gather*}
HHH, \qquad UDH, \qquad UHD, \qquad HUD.
\end{gather*}

A {\it generalized Motzkin path}, or simply $k$-Motzkin path, of 
size $n$ is a lattice path from $(0,0)$ to $(n,0)$ 
in the first quadrant such that 
each step is either an up step $U=(1,k)$, a down step $(1,-1)$,
or a horizontal path $(1,0)$.
\begin{example}
We have eleven $2$-Motzkin paths of size $5$:
\begin{gather*}
H^{5}, \qquad UDDHH, \qquad UDHDH, \qquad UHDDH, \qquad HUDDH, \\
UDHHD, \qquad UHDHD, \qquad HUDHD, \qquad UHHDD, \\ 
HUHDD, \qquad HHUDD.
\end{gather*}
\end{example}

\begin{defn}
The set of $k$-Motzkin paths of size $n$ is denoted by 
$\mathtt{Mot}(n;k)$.
\end{defn}

The generating function for Motzkin paths is given by  
\begin{align*}
\sum_{0\le n}|\mathtt{Mot}(n,1)|x^{n}
&=\genfrac{}{}{}{}{1-x-\sqrt{1-2x-3x^2}}{2x^2}, \\
&=1+x+2x^2+4x^3+9x^4+21x^5+51x^6+127x^{7}+\cdots.
\end{align*}
This integer sequence appears as A001006 in OEIS \cite{Slo}.
The generating function $y(x)$ of Motzkin paths satisfies $x^2y^2+(x-1)y+1=0$ where 
$y(x):=\sum_{0\le n}|\mathtt{Mot}(n,1)|x^{n}$.
The explicit formula for $|\mathtt{Mot}(n,1)|$ is given by 
\begin{align*}
|\mathtt{Mot}(n,1)|
=\genfrac{}{}{}{}{1}{n+1}
\sum_{j=0}^{n}\genfrac{(}{)}{0pt}{}{n+1}{n-j}\genfrac{(}{)}{0pt}{}{n-j}{j}.
\end{align*}

The number $|\mathtt{Mot}(n;k)|$ of $k$-Motzkin paths of 
size $n$ will be given in Proposition \ref{prop:kMot} by applying 
the Lagrange inversion theorem to the functional relation for 
$k$-Motzkin paths.

\subsection{Generalized Schr\"oder paths}
We first recall the definition of Schr\"oder paths, and generalize 
the notion of Schr\"oder paths to $k$-Schr\"oder paths.

We will introduce two types of generalizations of Schr\"oder path.
We call these two types type A and type B.

\subsubsection{Schr\"oder paths}
A {\it Schr\"oder path} of size $n$ is a lattice path from 
(0,0) to $(n,n)$ such that 
\begin{enumerate}
\item it consists of up step $U=(0,1)$, down step $D=(1,0)$, and 
horizontal step $H=(1,1)$,
\item it never goes below $y=x$.
\end{enumerate}
For simplicity, we write $U$, $D$ and $H$ for an up, down and horizontal
step respectively. 
We write a Schr\"oder path as a word in $\{U,D,H\}$.
\begin{defn}
We define $\mathtt{Sch}(n)$ as the set of Schr\"oder paths of size $n$.	
\end{defn}
The generating function of the numbers $|\mathtt{Sch}(n)|$ 
is given by 
\begin{align*}
\sum_{0\le n}|\mathtt{Sch}(n)|x^{n}&=\genfrac{}{}{}{}{1-x-\sqrt{x^2-6x+1}}{2x}, \\
&=1+2x+6x^2+22x^3+90x^4+394x^5+1806x^{6}+\cdots. 
\end{align*}
The number $|\mathtt{Sch}(n)|$ is well-known as large Schr\"oder numbers.
The generating function $y(x)=\sum_{0\le n}|\mathtt{Sch}(n)|x^{n}$ satisfies 
$xy^2-(1-x)y+1=0$.
This integer sequence appears as A006318 in OEIS \cite{Slo}.
The explicit expression of $|\mathtt{Sch}(n)|$ is given by 
\begin{align}
\label{eq:Sch}
|\mathtt{Sch}(n)|
=\genfrac{}{}{}{}{1}{n}\sum_{j=0}^{n}\genfrac{(}{)}{0pt}{}{n}{j}\genfrac{(}{)}{0pt}{}{n+j}{n-1}.
\end{align}

For example, we have six Schr\"oder paths of size $2$: 
\begin{gather*}
HH, \qquad HUD, \qquad UHD, \qquad UDH, \\
UUDD, \qquad UDUD.
\end{gather*}

\subsubsection{Generalized Schr\"oder paths of type $A$ and $B$}
Below, we generalized the notion of Schr\"oder paths by combining the notion
of $k$-Dyck paths with it.
We first introduce a $k$-Schr\"oder path of type $A$.
 
\begin{defn}[$k$-Schr\"oder paths of type $A$]
A {\it Fuss--Schr\"oder path of type A}, or simply {\it $k$-Schr\"oder path of type $A$} of size $n$ is 
a lattice path from $(0,0)$ to $(kn,n)$ such that 
\begin{enumerate}
\item it consists of an up step $U=(0,1)$, a down step $D=(1,0)$, and 
a horizontal step $H=(k,1)$, 
\item it never goes below $y=x/k$.
\end{enumerate}
We denote by $\mathtt{Sch}^{A}(n;k)$ the set of $k$-Schr\"oder paths of type A of size $n$.
\end{defn}

\begin{example}
\label{ex:SchA}
Set $(n,k)=(2,2)$. Then, we have eight $2$-Schr\"oder paths of type A of size $2$:
\begin{gather*}
HH, \qquad HUDD, \qquad UHDD, \qquad UDHD, \qquad UDDH, \\
UDDUDD, \qquad UDUDDD, \qquad UUDDDD.  
\end{gather*}
\end{example}

The lowest path $\mu_{0}$ and the highest path $\mu_1$ are given by
\begin{gather*}
\mu_{0}=H^{n}, \qquad \mu_{1}=U^{n}D^{kn}.
\end{gather*}

The number $|\mathtt{Sch}^{A}(n;k)|$ of $k$-Schr\"oder paths of type A 
of size $n$ is given in Proposition \ref{prop:kSchA} by use of the functional 
relation for $\mathtt{Sch}^{A}(n;k)$ and the Lagrange inversion theorem.

We give another generalization of Schr\"oder paths which we call type $B$.

\begin{defn}[$k$-Schr\"oder paths of type B]
A {\it Fuss--Schr\"oder path of type B}, or simply {\it $k$-Schr\"oder path 
of type B} of size $n$ is a lattice path from $(0,0)$ to $(kn,n)$ such that 
\begin{enumerate}
\item it consists of an up step $U=(0,1)$, a down step $D=(1,0)$ and 
a horizontal step $H=(1,1)$.
\item the path never goes below $y=x/k$.
\end{enumerate}

We denote by $\mathtt{Sch}^{B}(n;k)$ the set of $k$-Schr\"oder paths of type B of 
size $n$.
\end{defn}

The $k$-Schr\"oder paths of type B is studied in \cite{YanJia21} as $k$-Schr\"oder
paths.

\begin{example}
\label{ex:SchB}
Set $(n,k)=(2,2)$. We have ten $2$-Schr\"oder paths of type B:
\begin{gather*}
UDDUDD, \qquad HDUDD, \qquad UDDHD, \qquad HDHD, \\
UDUDDD, \qquad HUDDD, \qquad UDHDD, \qquad HHDD, \\
UUDDDD, \qquad UHDDD.
\end{gather*}
\end{example}
The lowest path $\mu_0$ and the highest path $\mu_1$ of type B are given by 
\begin{gather*}
\mu_{0}=(HD^{k-1})^{n}, \qquad \mu_{1}=U^{n}D^{kn}.
\end{gather*}
The cardinality of the set $\mathtt{Sch}^{B}(n;k)$ is given in 
Proposition \ref{prop:SchB} by use of the Lagrange inversion theorem.

\begin{remark}
Both types of $k$-Schr\"oder paths of size $n$ contain the set $\mathtt{Dyck}(n;k)$ 
as a subset. It is a subset of $k$-Schr\"oder paths without horizontal steps.
\end{remark}

Given a $k$-Schr\"oder path of type $B$, we denote by $Y(\mu_{0}/\mu)$ the 
skew shape surrounded by two paths $\mu_{0}$ and $\mu$.
Then, the area of $Y(\mu_{0}/\mu)$ is denoted by $|Y(\mu_{0}/\mu)|$.
Since we have a diagonal step, namely, a horizontal step, 
the value $|Y(\mu_{0}/\mu)|$ may be a half-integer.

For example, when $\mu=UHDDD$, we have 
$|Y(\mu_{0}/\mu)|=5/2$.

\begin{remark}
In the case of $k=1$, the both $k$-Schr\"oder paths of type A and type B 
give the same set of Schr\"oder paths. 
For a general $k\ge2$, Schr\"oder paths of type A and type B give 
different sets (see Example \ref{ex:SchA} and Example \ref{ex:SchB}).
\end{remark}

\subsection{Lagrange inversion theorem}
The Lagrange inversion theorem gives the inverse function of an analytic function.
In this paper, we use two forms of Lagrange inversion theorem.
\begin{theorem}[Lagrange inversion theorem, Theorem 1.2.4 in \cite{GouJac1983}]
\label{thrm:Lag}
Suppose that a function $f(x)$ satisfies the functional equation 
\begin{align*}
f(x)=1+xH(f(x)),
\end{align*}
where $H(x)$ is a polynomial of $x$. Then,
\begin{align*}
[x^{n}]f(x)=\genfrac{}{}{}{}{1}{n}[y^{n-1}]\left(H(1+y)\right)^{n}.
\end{align*}
Similarly, if $f(x)$ satisfies 
\begin{align*}
x=\genfrac{}{}{}{}{f(x)}{\phi(f(x))},
\end{align*}
where $\phi(x)$ is analytic with $\phi(0)\neq0$.
Then, we have 
\begin{align*}
[x^{n}]f(x)=\genfrac{}{}{}{}{1}{n}[y^{n-1}]\phi(y)^{n}.
\end{align*}
\end{theorem}

\section{Dimer models}
\label{sec:Dimer}
\subsection{Dimer model of Fibonacci type}
We first recall the definition of the multi-colored dimer model of Fibonacci type
following \cite{S21b}.
Let $v_{0}, v_{1},\ldots,v_{n}$ be vertices of a segment $[0,n]$.
A {\it dimer at position $i$} is an edge connecting two 
vertices $v_{i-1}$ and $v_{i}$.
A dimer configuration on a segment satisfies
\begin{enumerate}[(C1)]
\item there exists at most one dimer at position $i$,
\item if there exists a dimer $d$ with a color $c$ at position $i$, 
the dimers adjacent to $d$ have color $c'$ with $c'\neq c$.
\end{enumerate}
If dimers are adjacent to each other, we say that the dimers form 
a connected component. 
By definition, the distance between two connected components is at least 
one.
By definition of a dimer configuration, the maximum connected component
has $n$ dimers if we have $m\ge2$ colors, and has a single dimer if we have $m=1$ colors.

We denote by $\mathcal{D}_{n}$ the set of multi-colored dimer configurations 
on a segment $[0,n]$.

We introduce four formal variables associated to a dimer configuration.
\begin{defn}
The variable $m$ is the number of colors. 
The exponents of variables $r$, $s$, and $u$ represents the number of dimers, 
the sum of the positions of dimers, and the number of connected components, 
respectively.
\end{defn}

For the maximum connected component, we may have $n$ dimers with $n\ge1$.
We have $m(m-1)^{n-1}$ ways to put colors $n$ on dimers, and the weight 
of the component is given by $r^{n}$, $s^{n(n+1)/2}$ and $u^{1}$.
Thus in total, the maximum connected component gives the factor 
$m(m-1)^{n-1}r^{n}s^{n(n+1)/2}u^{1}$.

More in general, we define the generating function for dimer configurations 
$d\in\mathcal{D}_{n}$ on a segment.
\begin{defn}
Let $a$ be the number of connected components, $b$ be the number of dimers, 
and $c$ is the sum of the positions of dimers.
Then, we define the weight $\mathrm{wt}(d)$ of a dimer configuration $d$ as
\begin{align*}
\mathrm{wt}(d):=m^{a}(m-1)^{b-a}r^{b}s^{c}u^{a}.
\end{align*}
The generating function is defined as 
\begin{align*}
G_{n}(m,r,s,u):=\sum_{d\in\mathcal{D}_{n}}\mathrm{wt}(d).
\end{align*}
\end{defn}

\begin{example}
First few generating functions are  
\begin{align*}
G_{1}(m,r,s,u)&=1+mrsu, \\
G_{2}(m,r,s,u)&=1+mru(s+s^2)+m(m-1)r^{2}us^{3}, \\
G_{3}(m,r,s,u)&=1+rmu(s+s^2+s^3)+r^{2}m(m-1)u(s^{3}+s^{5})+r^2m^2u^2s^{4}.
\end{align*}
\end{example}

\begin{defn}
\label{defn:fg}
We define two functions $f(x)$ and $g(x)$ by
\begin{align*}
f(x):=1+x(m-1), \qquad g(x):=x(m(u-1)+1).
\end{align*}
\end{defn}

\begin{theorem}
\label{thrm:gfDyck}
The generating function $G_{n}(m,r,s)$ satisfies the 
recurrence relation of Fibonacci type 
\begin{align}
\label{eq:gfDyck}
G_{n}(m,r,s,u)=f(rs)G_{n-1}(m,rs,s,u)+g(rs)G_{n-2}(m,rs^2,s,u),
\end{align}
where $f(x)$ and $g(x)$ are defined in Definition \ref{defn:fg}.
The initial conditions are given by $G_{0}(m,r,s,u)=G_{-1}(m,r,s,u)=1$.	
\end{theorem}

To prove Theorem \ref{thrm:gfDyck}, we introduce two types of 
generating functions $z^{X}_{n}(m,r,s,u)$ with $X\in\{\emptyset,c\}$ 
which form $G_{n}(m,r,s,u)$.

We start from the definitions of $z^{\emptyset}$ and $z^{c}$.	
\begin{defn}
\label{defn:z}
Let $z^{\emptyset}_{n}(m,r,s,u)$ be the generating function of dimer 
configurations such that there exists no dimers at position one.
On the contrary, let $z^{c}_{n}(m,r,s,u)$ be the generating 
function of dimer configurations such that there exits a dimer of color $c$
at position one. 
\end{defn}

\begin{lemma}
\label{lemma:Gz}
The generating functions $G_{n}(m,r,s,u)$, $z^{\emptyset}_{n}(m,r,s,u)$ 
and $z^{c}_{n}(m,r,s,u)$ satisfy the following system of functional equations:
\begin{align}
\label{eq:Gninz}
&G_{n}(m,r,s,u)=z^{\emptyset}_{n}(m,r,s,u)+m z^{c}_{n}(m,r,s,u), \\
\label{eq:zemp}
&z^{\emptyset}_{n}(m,r,s,u)=G_{n-1}(m,rs,s,u), \\
\label{eq:zc}
&z^{c}_{n}(m,r,s,u)=rsu z^{\emptyset}_{n-1}(m,rs,s,u)+(m-1)rs z^{c}_{n-1}(m,rs,s,u).
\end{align}
\end{lemma}
\begin{proof}
The equality (\ref{eq:Gninz}) is obvious from the definitions of $z^{\emptyset}$ and $z^{c}$.
The factor $m$ comes from the number of possible colors at position one.

Observe that the shift $r\rightarrow rs$ increases the position of a dimer by one.
Thus, the equality (\ref{eq:zemp}) follows.

Since the distance of two connected components is at least one, $z^{c}_{n}$ 
satisfies Eq. (\ref{eq:zc}). Note that we have no $u$ in the second term of the right hand 
side of Eq. (\ref{eq:zc}) since the second term corresponds to increasing the size of 
the connected component by one.
\end{proof}

\begin{proof}[Proof of Theorem \ref{thrm:gfDyck}]
We will show that $G_{n}(m,r,s,u)$ satisfies the recurrence relation of Fibonacci type 
by use of Lemma \ref{lemma:Gz}.

From Eqs. (\ref{eq:Gninz}) and (\ref{eq:zemp}), we have 
\begin{align*}
m z^{c}_{n}(m,r,s,u)=G_{n}(m,r,s,u)-G_{n-1}(m,rs,s,u).
\end{align*}
From Eqs. from (\ref{eq:Gninz}) to (\ref{eq:zc}), the generating function $G_{n}(m,r,s,u)$ satisfies
\begin{align*}
G_{n}(m,r,s,u)&=z^{\emptyset}_{n}(m,r,s,u)+ m z^{c}_{n}(m,r,s,u), \\
&=G_{n-1}(m,rs,s,u)+mrsu G_{n-2}(m,rs^{2},s,u) \\
&\quad+(m-1)rs\left(G_{n-1}(m,rs,s,u)-G_{n-2}(m,rs^{2},s,u)\right), \\
&=f(rs)G_{n-1}(m,rs,s,u)+g(rs)G_{n-2}(m,rs^2,s,u),
\end{align*}
which completes the proof.
\end{proof}

\subsection{Generating function for empty dimers}
In this subsection, we consider the generating function of empty dimers.
Here, an empty dimer of position $i$ means that there is no dimer at position $i$.
We introduce two formal variables $p$ and $q$, where $p$ counts the total number 
of empty dimers, and the exponent of $q$ is the sum of the positions of 
empty dimers.
We keep the two variables $m$ and $u$.
We define $G_{n}^{\ast}(m,p,q)$ as a generating function of empty dimers 
satisfying the two conditions (C1) and (C2).
Note that we have a single color for empty dimers, however, a dimer 
may have $m$ possible colors.
The generating function $G_{n}^{\ast}(m,p,q)$ focuses on empty sites rather than 
dimers.

\begin{prop}
The generating function $G_{n}^{\ast}(m,p,q)$ satisfies the recurrence relation
\begin{align*}
G_{n}^{\ast}(m,p,q)=(pq+m-1)G_{n-1}^{\ast}(m,pq,q)+pq^{2}(m(u-1)+1)G_{n-2}^{\ast}(m,pq^2,q),
\end{align*}
with the initial conditions $G_{0}^{\ast}(m,p,q)=1$ and $G_{1}^{\ast}(m,p,q)=mu+pq$.
\end{prop}
\begin{proof}
Let $z_{n}^{\emptyset}(m,p,q)$ and $z_{n}^{c}(m,p,q)$ be two generating function for dimers as in 
Definition \ref{defn:z}.
Then, we have the following system of functional equations:
\begin{align*}
&G_{n}^{\ast}(m,p,q)=z_{n}^{\emptyset}(m,p,q)+mz_{n}^{c}(m,p,q), \\
&z_{n}^{\emptyset}(m,p,q)=pqG_{n-1}^{\ast}(m,pq,q), \\
&z_{n}^{c}(m,p,q)=upq^{2}G_{n-2}^{\ast}(m,pq^2,q)+(m-1)z_{n-1}^{c}(m,pq,q).
\end{align*}
Then, we have 
\begin{align*}
G_{n}^{\ast}(m,p,q)&=pq G_{n-1}^{\ast}(m,p,q)+mupq^2G_{n-2}(m,pq^2,q) \\
&\qquad+(m-1)\left(G_{n-1}^{\ast}(m,pq,q)-pq^{2}G_{n-2}^{\ast}(m,pq^2,q) \right), \\
&=(pq+m-1)G_{n-1}^{\ast}(m,pq,q)+pq^{2}(m(u-1)+1)G_{n-2}^{\ast}(m,pq^2,q),
\end{align*}
which completes the proof.
\end{proof}

\subsection{Generalized Dimer models with distance two}
Another generalization of the dimer model of Fibonacci type is 
to consider the models satisfying the following condition in addition to (C1) and (C2):
\begin{enumerate}
\item[(C3)] 
Two dimers with the same color have distance at least two.
\end{enumerate}

Let $G^{(2)}_{n}(m,r,s)$ be the generating function of the dimer models satisfying 
the conditions (C1), (C2) and (C3).
Note that the dimers with different colors can be adjacent to each other.
The size of connected components of dimers is one for $m=1$, two for $m=2$,
and arbitrary for $m\ge3$.

\begin{prop}
The generating function $G^{(2)}_{n}(m,r,s)$ satisfies the recurrence relation 
\label{prop:Gcdis2}
\begin{align*}
G^{(2)}_{n}(m,r,s)&=(1+rs(m-2))G^{(2)}_{n-1}(rs)+rs G^{(2)}_{n-2}(rs^2) \\
&\quad+rs(1+rs^2(m-1))G^{(2)}_{n-3}(rs^3)+r^2s^3G^{(2)}_{n-4}(rs^{4}),
\end{align*}
with the initial conditions 
\begin{align*}
G^{(2)}_{i}(m,r,s)=
\begin{cases}
1, & i\le 0, \\
1+rsm, & i=1.
\end{cases}
\end{align*}
\end{prop}
Before proceeding to the proof of Proposition \ref{prop:Gcdis2}, we introduce 
a lemma regarding generating functions.
We define two generating functions $z_{n}^{\emptyset}(m,r,s)$ and $z_{n}^{c}(m,r,s)$
satisfying the conditions (C1), (C2) and (C3)
as in Definition \ref{defn:z}.
\begin{lemma}
\label{lemma:Gcdis2}
The generating functions satisfy the following recurrence relations.
\begin{align*}
&G^{(2)}_{n}(m,r,s)=z_{n}^{\emptyset}(m,r,s)+m z_{n}^{c}(m,r,s), \\
&z_{n}^{\emptyset}(m,r,s)=G^{(2)}_{n-1}(m,rs,s), \\
&z_{n}^{c}(m,r,s)=rs\left((m-2)z_{n-1}^{c}(rs)+(m-1)z_{n-2}^{c}(rs^2)\right. \\
&\qquad\left.+rs^2(m-1)z_{n-3}^{c}(rs^3)
+G^{(2)}_{n-3}(rs^3)+rs^2G^{(2)}_{n-4}(rs^4)
\right)
\end{align*}
\end{lemma}
\begin{proof}
The first two equations are obvious from the definitions of generating functions.
We prove the third relation.

Recall that dimers with the same colors have distance at least two.
Thus, we have 
\begin{align}
\label{eq:recrelzc}
\begin{split}
z_{n}^{c}(r)&=rsG^{(2)}_{n-3}(rs^3)+rs(m-1)z_{n-2}^{c}(rs^2) \\
&\qquad+rs\left( (m-1) (z_{n-1}^{c}-rs^2(z_{n-2}^{c}(rs^2)-rs^3(z_{n-3}^{c}(rs^3)-\cdots  \right), \\
&=rsG^{(2)}_{n-3}(rs^{3})+rs(m-1)z_{n-2}^{c}(rs^2)+(m-1)\sum_{1\le j}(-1)^{j-1}r^js^{j(j+1)/2}z_{n-j}^{c}(rs^j).
\end{split}
\end{align}
The first term corresponds to the dimer configurations such that we have a dimer at position one and 
there are no dimers at positions two and three.
The second term corresponds to the configurations such that we have a dimer at position one and 
there is no dimer at position two, and there is a dimer at position three.
The third term comes from the configurations such that there is a connected component including 
the dimers at positions one and two.
Then, we have the following functional relation:
\begin{align*}
A_{n}(r)&:=\sum_{1\le j}(-1)^{j-1}r^{j}s^{j(j+1)/2}z_{n-j}^{c}(rs^j), \\
&=rsz_{n-1}^{c}(rs)-\sum_{2\le j}(-1)^{j}r^{j}s^{j(j+1)}z_{n-j}(rs^j), \\
&=rsz_{n-1}^{c}(rs)-rs\sum_{1\le j}(-1)^{j-1}(rs)^{j}s^{j(j+1)/2}z_{n-1-j}^{c}(rs^{j+1}), \\
&=rsz_{n-1}^{c}(rs)-rsA_{n-1}(rs).
\end{align*}
By rewriting the recurrence relation for $A_{n}(r)$ in terms of 
the generating functions $G^{(2)}_{n}(r)$ and $z^{c}_{n}$, 
we have the third recurrence relation.
\end{proof}

\begin{proof}[Proof of Proposition \ref{prop:Gcdis2}]
From the first two equations in Lemma \ref{lemma:Gcdis2}, we have 
\begin{align*}
mz_{n}^{c}(m,r,s)=G^{(2)}_{n}(m,r,s)-G^{(2)}_{n-1}(m,rs,s).
\end{align*}
By substituting this into the third equation in Lemma \ref{lemma:Gcdis2},
and rearranging the terms,
we obtain an expression of $z_{n}^{c}(m,r,s)$ in terms of only $G^{(2)}_{n}(m,r,s)$. 
We obtain the recurrence relation for $G^{(2)}_{n}(m,r,s)$ from the first equation 
in Lemma \ref{lemma:Gcdis2}.
This completes the proof.
\end{proof}

The first few generating functions $G_{n}^{(2)}(m,r,s)$ are 
\begin{align*}
G_{2}^{(2)}(m,r,s)&=1+r(ms+ms^2)+m(m-1)r^2 s^3, \\
G_{3}^{(2)}(m,r,s)&=1+mr(s+s^2+s^3)+m(m-1)r^2 (s^3 +s^4+s^5)+m(m-1)(m-2)r^3s^6, \\
G_{4}^{(2)}(m,r,s)&=1+mr(s+s^2+s^3+s^4)+r^2(m(m-1)(s^3+s^4+s^6+s^7)+m(2m-1)s^5) \\
&\quad +r^3(m(m-1)(m-2)(s^6+s^9)+m(m-1)^2(s^7+s^8))+m(m-1)(m-2)^2r^4s^{10}.
\end{align*}

\subsection{Type I}
We define a recurrence relation which we call type I.
We modify the recurrence relation (\ref{eq:gffib}) to preserve its
combinatorial structure. This means that the recurrence relation 
contains only terms $G_{n-p}(rs^{p})$ with a non-negative
integer $p$.

\begin{defn}
\label{defn:rrtypeI}
Consider the recurrence relation
\begin{align*}
G^{I}_{n}(m,r,s,u)=f(rs)G^{I}_{n-1}(m,rs,s,u)+g(rs)G^{I}_{n-k-1}(m,rs^{k+1},s,u),
\end{align*}
with initial conditions $G^{I}_{i}(m,r,s,u)=1$ for $i\le 0$.
We call this recurrence relation type I.
\end{defn}

Note that if we set $k=1$, the recurrence relation in Definition \ref{defn:rrtypeI}
is nothing but the recurrence relation (\ref{eq:gfDyck}) in Theorem \ref{thrm:gfDyck}.
Recall that the generating function of the dimer model satisfies Eq. (\ref{eq:gfDyck}) and 
it is related to Dyck paths \cite{S21b}. 
The generating function of type I preserves this combinatorial correspondence. More 
precisely, it has a description in terms of $k$-Dyck paths (see Section \ref{sec:glpI}).

\subsection{Type II}
We define a recurrence relation which we call type II.

We add a condition to the multi-colored dimer model on a segment.
\begin{enumerate}
\item[(C3$'$)] Let $d_1$ and $d_2$ be two connected components of dimers. 
The distance between $d_{1}$ and $d_{2}$ is at least $2$.
\end{enumerate}
We consider the dimer model with conditions (C1), (C2) and (C3$'$).

\begin{prop}
\label{prop:GIIk2}
The generating function $G^{II}_{n}(m,r,s,u)$ for the dimer model 
with conditions (C1), (C2) and (C3$'$) satisfies 
\begin{align}
\label{eq:GIIk2}
\begin{split}
G^{II}_{n}(m,r,s,u)&=f(rs)G^{II}_{n-1}(m,rs,s,u)-rs(m-1)G^{II}_{n-2}(m,rs^2,s,u) \\
&\qquad+mrsu G^{II}_{n-3}(m,rs^{3},s,u).
\end{split}
\end{align}
\end{prop}

Before proceeding to the proof of Proposition \ref{prop:GIIk2}, we introduce 
a lemma.
Define $z^{\emptyset}_{n}(m,r,s,u)$ and $z^{c}_{n}(m,r,s,u)$ as in Definition \ref{defn:z}
for the multi-colored dimer model with conditions (C1), (C2) and (C3$'$).

\begin{lemma}
\label{lemma:Gkdim}
The generating functions $G^{II}_{n}(m,r,s,u)$, $z^{\emptyset}_{n}(m,r,s,u)$ and $z^{c}_{n}(m,r,s,u)$
satisfy the following system of equations:
\begin{align}
\label{eq:GIIk2G}
&G^{II}_{n}(m,r,s,u)=z^{\emptyset}_{n}(m,r,s,u)+mz^{c}_{n}(m,r,s,u), \\
\label{eq:GIIk2e}
&z^{\emptyset}_{n}(m,r,s,u)=G^{II}_{n-1}(m,rs,s,u), \\
\label{eq:GIIk2c}
&z^{c}_{n}(m,r,s,u)=rsuG^{II}_{n-3}(m,rs^{3},s,u)+rs(m-1)z_{n-1}^{c}(m,rs,s,u).
\end{align}
\end{lemma}
\begin{proof}
The first two equations are obvious from the definition of generating functions.
The third equation reflects the condition that the distance between two connected 
components is at least $2$. Note that the first term corresponds to a dimer configuration 
such that it contains a dimer at position one, which is a connected component of size one,
and the second term corresponds to a configuration such that it has a connected component 
of size larger than one with a dimer at position one.
\end{proof}
	
\begin{proof}[Proof of Proposition \ref{prop:GIIk2}]
From Eqs. (\ref{eq:GIIk2G}) and (\ref{eq:GIIk2e}), we have 
\begin{align*}
mz_{n}^{c}(m,r,s,u)=G^{II}_{n}(m,r,s,u)-G^{II}_{n-1}(m,rs,s,u).
\end{align*}
Then, we have 
\begin{align*}
G^{II}_{n}(m,r,s,u)&=z^{\emptyset}_{n}(m,r,s,u)+mz^{c}_{n}(m,r,s,u), \\
&=G^{II}_{n-1}(m,rs,s,u)+mrsuG^{II}_{n-3}(m,rs^3,s,u)+rsm(m-1)z_{n-1}^{c}(m,rs,s,u), \\
&=(1+rs(m-1))G^{II}_{n-1}(m,rs,su)-rs(m-1)G^{II}_{n-2}(m,rs^2,s,u) \\
&\qquad+mrsuG^{II}_{n-3}(m,rs^3,s,u),
\end{align*}
which completes the proof.
\end{proof}

\begin{remark}
\label{remark:typeII}
From the recurrence relation (\ref{eq:GIIk2}), we observe that the sum of the coefficients 
of the second and the third terms in the right had side is equal to $g(rs)$.
This implies that we split the third term in Definition \ref{defn:rrtypeI} into 
two terms to reflect the effect of the separation of two connected components by two sites.
\end{remark}

We define a series of recurrence relations by generalizing the 
recurrence relation (\ref{eq:GIIk2}).
\begin{defn}
\label{defn:GnII}
Let $k\ge2$. Consider the recurrence relation 
\begin{align*}
G^{II}_{n}(m,r,s,u)&=f(rs)G^{II}_{n-1}(m,rs,s,u)-rs(m-1)G^{II}_{n-k}(m,rs^{k},s,u) \\
&\qquad+mrsuG^{II}_{n-k-1}(m,rs^{k+1},s,u),
\end{align*}
with initial conditions $G^{II}_{i}(m,r,s,u)=1$ for $i\le 0$.
We call this recurrence relation type II.
\end{defn}
When $k=2$, the type II relation is nothing but 
the recurrence relation (\ref{eq:GIIk2}).

Similarly, we consider another recurrence relation of type II:
\begin{defn}
\label{defn:HII}
The recurrence relation for $H^{II}_{n}(m,r,s,u)$ is defined to be
\begin{align*}
H^{II}_{n}(m,r,s,u)=f(rs)H^{II}_{n-1}(m,rs,s,u)-rs(m-1)H^{II}_{n-2}(m,rs^2,s,u)
+mrsu H^{II}_{n-k-1}(rs^{k+1})
\end{align*}
for $k\ge1$.
\end{defn} 
The generating function $H^{II}_{n}(m,r,s,u)$ comes from the dimer model 
where the distance of two adjacent connected components is at least $k$ sites.
When $k=2$, the recurrence relations for $G^{II}_{n}$ and $H^{II}_{n}$ coincide
with each other.

\subsection{Type III}
In type I and type II, we consider the dimer model where the length of 
a dimer is one.
In this subsection, we consider the dimers of length $k$.

A dimer of length $k$ can be obtained from $k$ dimers of length one by 
gluing them together.
We define the position of a dimer $d$ of length $k$ as the position of 
the leftmost dimer of length one which is contained by the dimer $d$. 
Though a dimer of length $k$ consists of $k$ dimers of length one, 
we count a dimer of length $k$ as a single dimer.
A glued dimer of length $k$ has a single color, i.e., all the dimers of length one 
which constitute a glued dimer of length $k$ have the same color.

We first consider the $k=2$ case, and generalize it to the $k\ge3$ case.

\begin{prop}
\label{prop:GIIIk2}
The generating function $G^{III}_{n}(r):=G^{III}_{n}(m,r,s,u)$ of the dimers of length $2$ 
satisfies the following recurrence relation: 
\begin{align}
\label{eq:GIIIk2}
G^{III}_{n}(r)=G^{III}_{n-1}(rs)+rs(m-1)G^{III}_{n-2}(rs^{2})+rs(m(u-1)+1)G^{III}_{n-3}(rs^{3}).
\end{align}
\end{prop}
\begin{proof}
Let $z^{\emptyset}(r)$ and $z^{c}(r)$ be the generating function of dimer configurations 
as in Definition \ref{defn:z}.
Then, the three generating functions $G^{III}_{n}(r)$, $z^{\emptyset}(r)_{n}$ and $z^{c}_{n}(r)$
satisfy 
\begin{align*}
G^{III}_{n}(r)&=z^{\emptyset}_{n}(r)+mz^{c}_{n}(r), \\
z_{n}^{\emptyset}(r)&=G^{III}_{n-1}(rs), \\
z_{n}^{c}(r)&=rs\left( uz_{n-2}^{\emptyset}(rs^2)+(m-1)z_{n-2}^{c}(rs^2) \right).
\end{align*}
In the third equation, we have used the fact that the position of a dimer of length two 
is the position of the left-most dimer of length one.
From these, we have
\begin{align*}
G^{III}_{n}(r)&=z^{\emptyset}(r)+mz^{c}(r), \\
&=G^{III}_{n-1}(rs)+rs(muG^{III}_{n-3}(rs^3)+(m-1)G^{III}_{n-2}(rs^2)-(m-1)G^{III}_{n-3}(rs^3)),\\
&=G^{III}_{n-1}(rs)+rs(m-1)G^{III}_{n-2}(rs^2)+rs(m(u-1)+1)G^{III}_{n-3}(rs^3),
\end{align*}
which completes the proof.
\end{proof}

\begin{remark}
The sum of the coefficients of the first and the second terms in the right hand 
side of Eq. (\ref{eq:GIIIk2}) is equal to $f(rs)$, and the coefficient of the third 
term is $g(rs)$.
This implies that the split of the second term in Definition \ref{defn:rrtypeI} into 
two terms corresponds to increase of the size of dimers.
See also Remark \ref{remark:typeII} for a similar phenomenon for $g(rs)$.
\end{remark}

We define the generating function $G^{III}_{n}(r)$ with $k\ge1$ as follows.
\begin{defn}
\label{defn:GIII}
The recurrence relation of type III is defined as 
\begin{align*}
G^{III}_{n}(r)=G^{III}_{n-1}(rs)+rs(m-1)G^{III}_{n-k}(rs^{k})+rs(m(u-1)+1)G^{III}_{n-k-1}(rs^{k+1}).
\end{align*}
\end{defn}
The generating function $G_{n}^{III}(r)$ is the generating function of glued dimers of length $k$.
Note that we have recurrence relation (\ref{eq:GIIIk2}) in Proposition \ref{prop:GIIIk2}
by setting $k=2$.
Similarly, if we set $k=1$, the recurrence relation (\ref{eq:GIIIk2}) is nothing but the 
recurrence relation (\ref{eq:gfDyck}) of Fibonacci type.

\section{Generating functions for generalized lattice paths}
\label{sec:gfGLP}
For simplicity, we define $m'$ and $u'$ as follows.
\begin{defn}
\label{defn:mudash}
We define $m':=m-1$ and $u':=u-1$.
\end{defn}
In this section, we study the functional relation arising from 
the generating functions of dimer models in the previous section, 
and show that it has an explicit formula in terms of generalized 
lattice paths introduced in Section \ref{sec:LP}.

\subsection{Dyck paths}
We first recall the correspondence between the generating function 
of the dimer model and Dyck paths.
The recurrence relation (\ref{eq:gfDyck}) gives the following functional equation 
for a formal power series $\chi(m,r,s,u)$.
\begin{align}
\label{eq:defchi}
f(rs)\chi(m,r,s,u)=1-g(rs)\chi(m,r,s,u)\chi(m,rs,s,u),
\end{align}
where $f(x)$ and $g(x)$ are defined in Definition \ref{defn:fg}.

Let $\mu$ be a Dyck path expressed in terms of $U$ and $D$, and $\mathrm{Peak}(\mu)$ be the number 
of peaks, that is, the number of a pattern $UD$ in $\mu$.
\begin{defn}
We define the weight for $\mu$ by
\begin{align}
\label{eq:defwtmu}
\mathrm{wt}(\mu):=(m'+1)^{\mathrm{Peak}(\mu)}(u'+1)^{\mathrm{Peak}(\mu)}(m'u'+u'+1)^{n-\mathrm{Peak}(\mu)}.
\end{align}
\end{defn}

\begin{prop}
The formal power series $\chi(m,r,s,u)$ is given by
\begin{align*}
\chi(m,r,s,u)=\sum_{0\le n}(-rs)^{n}\sum_{\mu\in\mathtt{Dyck}(n)}s^{|Y(\mu_{0}/\mu)|}\mathrm{wt}(\mu),
\end{align*}
where $|Y(\mu_{0}/\mu)|$ is the number of unit boxes above $\mu_{0}$ and below $\mu$.
\end{prop}
\begin{proof}
Let 
\begin{align*}
\chi_{n}(m',s,u'):=\sum_{\mu\in\mathtt{Dyck}(n;k)}s^{|Y(\mu_{0}/\mu)|}\mathrm{wt}(\mu).
\end{align*}
From Eq. (\ref{eq:defchi}), $\chi_{n}(s):=\chi_{n}(m',s,u')$ should satisfy 
\begin{align*}
\chi_{n}(s)&=m'\chi_{n-1}(s)+(m'u'+u'+1)\sum_{j=0}^{n-1}s^{n-1-j}\chi_{j}(s)\chi_{n-1-j}(s), \\
&=(m'u'+m'+u'+1)\chi_{n-1}(s)+(m'u'+u'+1)\sum_{j=0}^{n-2}s^{n-1-j}\chi_{j}(s)\chi_{n-1-j}(s), \\
&=\chi_{1}(s)\chi_{n-1}(s)+(m'u'+u'+1)\sum_{j=0}^{n-2}s^{n-1-j}\chi_{j}(s)\chi_{n-1-j}(s),
\end{align*}
where we have used $\chi_{1}(s)=(m'+1)(u'+1)$ in the last equality.
This recurrence relation for $\chi_{n}(s)$ can be verified by a simple calculation, 
which completes the proof.
\end{proof}

We expand $\chi(m,r,s=1,u)$ as 
\begin{align*}
\chi(m,r,s=1,u)=\sum_{0\le n}(-r)^{n}S_{1}(n; m',u').
\end{align*}
Specializations of the variables $m'$ and $u'$ give several enumerations of lattice paths 
regarding $S_{1}(n; m',u')$.
We consider four specializations $m',u'\in\{0,1\}$.

\begin{prop}
\label{prop:Dycksp}
The number $S_{1}(n; m',u')$ with $m',u'\in\{0,1\}$ enumerates the following combinatorial objects.
\begin{enumerate}
\item  The number $S_{1}(n;0,0)=\mathtt{Dyck}(n;1)$. 
\item The number $S_{1}(n;0,1)$ is the number of Dyck paths of size $n$ with two colors for up steps.
Thus the cardinality of $S_{1}(0,1)$ is given by $2^{n}|\mathtt{Dyck}(n;1)|$.
\item The number $S_{1}(n;1,0)$ is given by the number of Dyck paths with two colors $\{U,u\}$ for an up step 
and single color $D$ which avoid the pattern $Du$.
Thus, $S_{1}(n;1,0)$ is given by large Schr\"oder number.
\item The number $S_{1}(n;1,1)$ is given by the number of Dyck paths with two colors $\{U,u\}$ for an up step
and two colors $\{D,d\}$ for a down step which avoid the pattern $ud$.
\end{enumerate}
\end{prop}

The relations among $S_{1}(n;m',u')$, $m',u'\in\{0,1\}$ and the sequence numbers in OEIS \cite{Slo}
are summarized below. 
\begin{gather*}
\begin{CD}
\raisebox{-0.4\totalheight}{\shortstack{$S_{1}(n;0,1)$ \\ A151374}} 
@<{m'=0}<< \raisebox{-0.4\totalheight}{\shortstack{$S_{1}(n;1,1)$ \\ A103211}}  \\
@V{u'=0}VV @VV{u'=0}V \\
\raisebox{-0.4\totalheight}{\shortstack{$S_{1}(n;0,0)$ \\ A000108}} 
@<{m'=0}<< \raisebox{-0.4\totalheight}{\shortstack{$S_{1}(n;1,0)$ \\ A006318}} 	
\end{CD}
\end{gather*}

\begin{proof}[Proof of Proposition \ref{prop:Dycksp}]
(1). From the definition (\ref{eq:defwtmu}) of weight, we have 
$\mathrm{wt}(\mu)=1$ for $m'=u'=0$.
Therefore, it is obvious that 
the value $S_{1}(n;0,0)$ is equal to the number of Dyck paths of size $n$.

(2). When $m'=0$ and $u'=1$, we have $\mathrm{wt}(\mu)=2^{n}$. 
Note that the system size $n$ is equal to the number of up steps.
Thus, the cardinality of $S_{1}(n;0,1)$ is given by $2^{n}|\mathtt{Dyck}(n,1)|$.	

(3). When $m'=1$ and $u'=0$, we have $\mathrm{wt}(\mu)=2^{\mathrm{Peak}(\mu)}$.
This means that we have two choices for each peak.
Since a peak corresponds to the pattern $UD$ in a Dyck path, we may replace the peak $UD$ 
by a horizontal step $H=(1,1)$. The newly obtained path never goes below $y=x$.
The new path is nothing but a Schr\"oder path. 
Therefore, the cardinality of $S_{1}(n;1,0)$ is given by large Schr\"oder numbers.	

We count the Dyck paths which satisfy the conditions in (3).
Let $\gamma(r):=\sum_{n\ge0}A_nr^{n}$ be the generating function such that $A_{n}$ is the 
number of such Dyck paths of size $n$. We show that $\gamma(r)$ satisfies the same recurrence 
relation as the one for Schr\"oder paths.
We have two Dyck paths of size one: $UD$ and $ud$.
For $n\ge2$, we insert Dyck paths between $U$ and $D$ or right to $D$.
Note that the first up step in a Dyck path has a color $U$ or $u$.
Since a Dyck path is a prime or not, we have a recurrence relation for $\gamma(r)$:
\begin{align*}
\gamma(r)&=1+r(2\gamma(r)+\gamma(r)(\gamma(r)-1)), \\
&=1+r(\gamma(r)+\gamma(r)^2).
\end{align*}
It is easy to see that this recurrence relation is the same as the one for Schr\"oder paths.

(4). 	When $m'=u'=1$, we have 
\begin{align}
\label{eq:wtval}
\begin{split}
\mathrm{wt}(\mu)&=4^{\mathrm{Peak}(\mu)}3^{n-\mathrm{Peak}}, \\
&=4^{\mathrm{Val}(\mu)+1}3^{n-1-\mathrm{Val}(\mu)},
\end{split}
\end{align}
where $\mathrm{Val}(\mu)$ is the number of valleys in $\mu$, that is, the number 
of $DU$ in $\mu$. Note that $\mathrm{Val}(\mu)=\mathrm{Peak}(\mu)-1$.
Recall that a Narayana number $\mathtt{Nar}(n,k)$ is the number of Dyck paths of size 
$n$ with $k$ peaks. 
By the symmetry $k\leftrightarrow n+1-k$, Narayana numbers satisfy 
$\mathtt{Nar}(n,k)=\mathrm{Nar}(n,n+1-k)$.
Thus, the number of Dyck paths of size $n$ with $p$ valleys is equal to 
the one of Dyck paths with $n-1-p$ valleys.
Therefore, we have, from Eq. (\ref{eq:wtval}) 
\begin{align*}
\mathrm{wt}(\mu)=4^{n-\mathrm{Val}(\mu)}3^{\mathrm{Val}(\mu)}.
\end{align*}
Since $3=4-1$, this expression implies that the number $|S_{1}(n;1,1)|$ is the number 
of Dyck paths which have two colors  $\{U,u\}$ and $\{D,d\}$ on each up and down step 
respectively and avoid the pattern $ud$.
\end{proof}

\subsection{Recurrence relation for empty dimers}
We study a formal power series coming from the recurrence relation for empty dimers, 
and show that it has formulae in terms of $k$-Dyck paths and $k$-Schr\"oder paths 
of type B.
\begin{defn}
\label{defn:chiast}
Let $\chi^{\ast}(p,q)$ be a formal power series satisfying the following 
recurrence relation:
\begin{align}
\label{eq:chiast0}
1-(pq+m-1)\chi^{\ast}(p,q)-pq^2(m(u-1)+1)\chi^{\ast}(p,q)\chi^{\ast}(pq,q)=0.
\end{align}
\end{defn}

We define a weight for a Dyck path $\mu\in\mathtt{Dyck}(n;1)$.
Recall that any Dyck path $\mu$ is written as a concatenation of 
two paths $\mu_{1}\circ\mu_{2}$ such that 
$\mu_{1}$ is a prime path.
We denote by $|\mu|$ the size of Dyck path $\mu$.
We define the weight for a Dyck paths recursively as follows.

We fist define the weight for a zig-zag Dyck path, {\it i.e.}, 
$\mu=(UD)^{n}$:
\begin{align*}
\mathrm{wt}^{\ast}((UD)^{n}):=(m'+q(m'u'+u'+1))^{n},
\end{align*}
where $m'$ and $u'$ are defined in Definition \ref{defn:mudash}.
Then, for a general Dyck path $\mu:=\mu_1\circ\mu_2$ with $\mu_1$ prime, 
we define 
\begin{align*}
\mathrm{wt}^{\ast}(\mu):=\mathrm{wt}^{\ast}(\mu_1)\mathrm{wt}^{\ast}(\mu_2).
\end{align*}
Let $\widetilde{\mu}_{1}$ be a Dyck path of size $|\mu_1|-1$ 
obtained from a prime Dyck path $\mu_1$ by deleting the first up step and the last 
down step.
Then, we define 
\begin{align*}
\mathrm{wt}^{\ast}(\mu_{1}):=q^{|\mu_{1}|}(m'u'+u'+1)\mathrm{wt}^{\ast}(\widetilde{\mu}_{1}).
\end{align*}

\begin{prop}
The formal power series $\chi^{\ast}(p,q)$ is expressed in terms of 
Dyck paths:
\begin{align}
\label{eq:chiast}
\chi^{\ast}(p,q)=\sum_{0\le n}\genfrac{}{}{}{}{(-pq)^{n}}{(m-1)^{2n+1}}
\sum_{\mu\in\mathtt{Dyck}(n;1)}\mathrm{wt}^{\ast}(\mu).
\end{align}
\end{prop}
\begin{proof}
Let $\widetilde{\chi}^{\ast}(p,q)=\chi^{\ast}(-p,q)$ given by Eq. (\ref{eq:chiast}).
Then, the generating function $\widetilde{\chi}^{\ast}(p,q)$ satisfies the following 
functional equation:
\begin{align*}
\widetilde{\chi}^{\ast}(p,q)&=\genfrac{}{}{}{}{1}{m'}+\genfrac{}{}{}{}{1}{m'^{2}}(m'+q(m'u'+u'+1))pq\widetilde{\chi}^{\ast}(p,q) \\
&\qquad+\genfrac{}{}{}{}{1}{m'}(m'u'+u'+1)pq^2(\widetilde{\chi}^{\ast}(pq,q)-\genfrac{}{}{}{}{1}{m'})\widetilde{\chi}^{\ast}(p,q), \\
&=\genfrac{}{}{}{}{1}{m'}+\genfrac{}{}{}{}{pq}{m'}\widetilde{\chi}^{\ast}(p,q)
+\genfrac{}{}{}{}{1}{m'}(m'u'+u'+1)pq^2\widetilde{\chi}{\ast}(pq,q)\widetilde{\chi}^{\ast}(p,q).
\end{align*}
The first term in the first line corresponds to the empty path and 
the second term corresponds to a path $\mu=\mu_{1}\circ\mu_{2}$ such that 
$\mu_{1}=UD$.
The third term in the second line corresponds to a path $\mu=\mu_{1}\circ\mu_{2}$ 
such that $\mu_1$ is prime and $|\mu_{1}|\ge 2$.
Then, it is obvious that $\chi^{\ast}(p,q)=\widetilde{\chi}^{\ast}(-p,q)$ satisfies 
the recurrence relation (\ref{eq:chiast0}). 
This completes the proof.
\end{proof}

\begin{prop}
Set $q=1$. The formal power series $\chi^{\ast}(p,q=1)$ is given
by
\begin{align}
\label{eq:chiastq1}
[p^{n}]\chi^{\ast}(p,q=1)
=\genfrac{}{}{}{}{(-1)^{n}}{n}
\sum_{t=0}^{n}\genfrac{}{}{}{}{(m'u'+u'+1)^{n-t}}{m'^{2n+1-t}}
\genfrac{(}{)}{0pt}{}{n}{t}\genfrac{(}{)}{0pt}{}{2n-t}{n-1}.
\end{align}
\end{prop}
\begin{proof}
From the recurrence relation (\ref{defn:chiast}) with $q=1$, 
we have 
\begin{align*}
\chi'(p)=1+p H(\chi'(p)),
\end{align*}
where $\chi'(p)=m'\chi^{\ast}(p)$ and
\begin{align*}
H(x):=-\genfrac{}{}{}{}{1}{m'}x-\genfrac{}{}{}{}{m'u'+u'+1}{m'^2}x^{2}.
\end{align*}
We apply Lagrange inversion theorem to $H(x)$, 
and obtain Eq. (\ref{eq:chiastq1}).
\end{proof}

We generalize the above result to the case of $k$-Dyck paths.
For this, we modify the recurrence relation in Definition \ref{defn:chiast} 
to fit the case of $k$-Dyck paths.
The formal power series has also an expression in terms of $k$-Schr\"oder 
paths of type B.
\begin{defn}
Fix $k\ge1$.
Let $\chi_{k}^{\ast}(p,q)$ be a formal power series satisfying the following 
recurrence relation:
\begin{align}
\label{eq:chiempty}
1-(m-1)\chi_{k}^{\ast}(p,q)-pq\prod_{j=0}^{k-1}\chi_{k}^{\ast}(pq^j,q)-pq^2(m(u-1)+1)\prod_{j=0}^{k}\chi_{k}^{\ast}(pq^j,q)=0.
\end{align} 
\end{defn}

We define the weight for a $k$-Dyck path of size $n$.
\begin{align*}
\mathrm{wt}^{\ast}((UD^{k})^{n}):=(m'+q(m'u'+u'+1))^n
\end{align*}
When a $k$-Dyck path $\mu$ is written as a concatenation of two 
$k$-Dyck paths, $\mu:=\mu_{1}\circ\mu_{2}$ with $\mu_1$ prime, 
the weight of $\mu$ is defined as 
\begin{align*}
\mathrm{wt}^{\ast}(\mu):=\mathrm{wt}^{\ast}(\mu_1)\mathrm{wt}^{\ast}(\mu_{2}).
\end{align*} 
Let $\mu_{1}$ be a prime $k$-Dyck path.
Then $\mu_1$ has the following form:
\begin{align*}
\mu_{1}=U\circ \lambda_{1}\circ D \circ \lambda_{2} \circ \ldots \circ D\circ \lambda_{k}\circ D,
\end{align*}
where paths $\lambda_{i}$, $1\le i\le k$, are $k$-Dyck paths.
The weight $\mathrm{wt}^{\ast}(\mu_1)$ is defined by
\begin{align}
\label{eq:wtchiDyck}
\begin{split}
\mathrm{wt}^{\ast}(\mu_1):=
\begin{cases}
\displaystyle
q(m'u'+u'+1)\prod_{i=1}^{k}q^{(k+1-i)|\lambda_{i}|}\mathrm{wt}^{\ast}(\lambda_{i}), & \lambda_{1}\neq\emptyset, \\
\displaystyle
(m'+q(m'u'+u'+1))\prod_{i=2}^{k}q^{(k+1-i)|\lambda_{i}|}\mathrm{wt}^{\ast}(\lambda_{i}), & \lambda_{1}=\emptyset.
\end{cases}
\end{split}
\end{align}
Note that if we set $k=1$, the weight defined above coincides with the definition in the case of Dyck paths.	

\begin{prop}
\label{prop:chiinDyck}
The formal power series $\chi_{k}^{\ast}(p,q)$ is given by 
\begin{align}
\label{eq:empchi}
\chi_{k}^{\ast}(p,q)
=\sum_{0\le n}\genfrac{}{}{}{}{(-pq)^{n}}{(m-1)^{(k+1)n+1}}
\sum_{\mu\in\mathtt{Dyck}(n;k)}\mathrm{wt}^{\ast}(\mu)q^{\mathrm{Area}(\mu)},
\end{align}
where $\mathrm{Area}(\mu)$ is the number of unit boxes below $\mu$ and above the lowest path $\mu_{0}$.
\end{prop}
\begin{proof}
Let $\widetilde{\chi}^{\ast}(p,q):=\chi_{k}^{\ast}(-p,q)$ given by Eq. (\ref{eq:empchi}).
Then, as in the case of $k=1$, the formal power series $\widetilde{\chi}^{\ast}(p,q)$ satisfies the following recurrence relation:
\begin{align*}
\widetilde{\chi}^{\ast}(p,q)&=\genfrac{}{}{}{}{1}{m'}
+\genfrac{}{}{}{}{1}{m'^{2}}(m'+q(m'u'+u'+1))pq\prod_{j=0}^{k-1}\widetilde{\chi}^{\ast}(pq^{j},q) \\
&\qquad+\genfrac{}{}{}{}{1}{m'}(m'u'+u'+1)pq^2(\widetilde{\chi}^{\ast}(pq^{k},q)
-\genfrac{}{}{}{}{1}{m'})\prod_{j=0}^{k-1}\widetilde{\chi}^{\ast}(pq^{j},q), \\
&=\genfrac{}{}{}{}{1}{m'}+\genfrac{}{}{}{}{pq}{m'}\prod_{j=0}^{k-1}\widetilde{\chi}^{\ast}(pq^{j},q)
+\genfrac{}{}{}{}{1}{m'}(m'u'+u'+1)pq^2\prod_{j=0}^{k}\widetilde{\chi}{\ast}(pq^{j},q).
\end{align*}
It is obvious that the formal power series $\chi_{k}^{\ast}(p,q)$ satisfies Eq. (\ref{eq:chiempty}).
The expression (\ref{eq:empchi}) is a formal power series for the recurrence relation (\ref{eq:chiempty}).
\end{proof}

The formal power series $\chi_{k}^{\ast}(p,q)$ has another 
expression in terms of $k$-Schr\"oder paths of type B.

Let $\mu$ be a $k$-Schr\"oder path of type B of size $n$.
We denote by $\widetilde{\mu}$ a $k$-Dyck path obtained 
from $\mu$ by replacing a $H$ by a $UD$.

\begin{prop}
The formal power series $\chi_{k}^{\ast}$ satisfies 
\begin{align}
\label{eq:chiemptySchB}
\chi_{k}^{\ast}(p,q)=\sum_{0\le n}\genfrac{}{}{}{}{(-pq)^{n}}{(m-1)^{(k+1)n+1}}
\sum_{\mu\in\mathtt{Sch}^{B}(n;k)}
m'^{N(H)}(m'u'+u'+1)^{N(U)}q^{d(\mu)},
\end{align}
where $N(H)$ and $N(U)$ are the numbers of horizontal and up steps in $\mu$ respectively, 
and 
\begin{align}
\label{eq:SchArea}
d(\mu):=\mathrm{Area}(\widetilde{\mu})+n-N(H).
\end{align}
\end{prop}
\begin{proof}
From Proposition \ref{prop:chiinDyck}, it is enough to show 
that 
\begin{align}
\label{eq:chiDyckSch}
\sum_{\mu\in\mathtt{Dyck}(n;k)}\mathrm{wt}^{\ast}(\mu)q^{\mathrm{Area}(\mu)}
=\sum_{\lambda\in\mathtt{Sch}^{B}(n;k)}m'^{N(H)}(m'u'+u'+1)^{N(U)}q^{d(\lambda)}.
\end{align}
Since the factors in Eq. (\ref{eq:wtchiDyck}) consist of $q$, $(m'+u'+1)$, and 
$(m'+q(m'u'+u'+1))$, the expansion of the left hand side of Eq. (\ref{eq:chiDyckSch})
in terms of the factors $q$, $m'$ and $(m'+u'+1)$ is unique.
We first observe a correspondence between a $k$-Dyck path and a set of $k$-Schr\"oder 
paths of type B.
Given a $k$-Dyck path $\mu$, we have a set of paths $\lambda\in\mathtt{Sch}^{B}(n;k)$ 
obtained from $\mu$ by replacing $UD$ with $H$.
When the path $\mu$ has $N_{UD}$ $UD$ patterns, the cardinality of the set is 
$2^{N_{UD}}$.
From the definition of the weight (\ref{eq:wtchiDyck}) given to $\mu$, 
we give the weight $m'$ to a horizontal step in $\lambda$ and the weight 
$m'u'+u'+1$ to an	 up step in $\lambda$.

We compare the exponent of $q$ in the right hand side of Eq. (\ref{eq:chiDyckSch}) with 
that in the left hand side of Eq. (\ref{eq:chiDyckSch}).
Note that the definition of the weight (\ref{eq:wtchiDyck}) reflects the 
existence of the pattern $UD$ in $\mu$. 
This implies that the exponent of $q$ in $\mathrm{wt}^{\ast}(\mu)$ is equal to 
the sum of the number $n-N(H)$ and the area $\mathrm{Area}(\mu)$ in $\lambda$. 
Therefore, $\lambda\in\mathtt{Sch}^{B}(n;k)$,  
we have  
\begin{align}
\label{eq:SchArea1}
d(\lambda)=\mathrm{Area}(\widetilde{\lambda})+n-N(H),
\end{align}
where $N(H)$ is the number of horizontal steps in $\lambda$ and $\widetilde{\lambda}$ is the $k$-Dyck path 
which is obtained from $\lambda$ by replacing $H$ by $UD$.

From these observations, we have 
the correspondence (\ref{eq:chiDyckSch}), which completes 
the proof.
\end{proof}

\begin{prop}
\label{prop:SchB}
The number of $k$-Schr\"oder paths of type $B$ of size $n$ is 
given by
\begin{align}
\label{eq:SchBnk}
|\mathtt{Sch}^{B}(n;k)|
=\genfrac{}{}{}{}{1}{n}
\sum_{m=0}^{n-1}2^{m+1}\genfrac{(}{)}{0pt}{}{nk}{m}\genfrac{(}{)}{0pt}{}{n}{n-m-1}.
\end{align}
Similarly, we also have another expression similar to Eq. (\ref{eq:Sch}):
\begin{align}
\label{eq:SchBnk2}
|\mathtt{Sch}^{B}(n;k)|
=\genfrac{}{}{}{}{1}{n}\sum_{j=0}^{n}\genfrac{(}{)}{0pt}{}{n}{j}\genfrac{(}{)}{0pt}{}{kn+j}{n-1}.
\end{align}
\end{prop}
\begin{proof}
Set $(m',u')=(1,0)$ and $(p,q)=(-p,1)$ in Eq. (\ref{eq:chiempty}).
Then, $\chi:=\chi^{\ast}_{k}(p)$ satisfies 
\begin{align}
\label{eq:chiSchB}
\chi=1+p(\chi^{k}+\chi^{k+1}).
\end{align}
On the other hand, from the expression (\ref{eq:chiemptySchB}), 
we have 
\begin{align*}
[p^{n}]\chi=|\mathtt{Sch}^{B}(n;k)|.
\end{align*}
By applying the Lagrange inversion theorem to Eq. (\ref{eq:chiSchB}), 
we obtain
\begin{align*}
[p^{n}]\chi&=\genfrac{}{}{}{}{1}{n}[y^{n-1}]\left((1+y)^{k}(2+y)\right)^{n}, \\
&=
\genfrac{}{}{}{}{1}{n}\sum_{m=0}^{n-1}2^{m+1}\genfrac{(}{)}{0pt}{}{nk}{m}\genfrac{(}{)}{0pt}{}{n}{n-m-1},
\end{align*}
which completes the proof.

Equation (\ref{eq:SchBnk2}) is similarly shown by using the Lagrange inversion 
theorem. We use the relation 
\begin{align*}
(2+y)^{n}=\sum_{j=0}^{n}\genfrac{(}{)}{0pt}{}{n}{j}(1+y)^{j},
\end{align*}
and rearrange the terms.
\end{proof}

The integer sequences $|\mathtt{Sch}^{B}(n;k)|$ appear in OEIS \cite{Slo}
as A006318, A027307, A144097 and A260332 for $k=1,2,3$ and $4$ respectively.

\begin{cor}
Let $0\le m\le m-1$.
The number of $k$-Dyck paths of size $n$ with $m+1$ peaks is given by 
\begin{align}
\label{eq:kDyckmpeaks}
\genfrac{}{}{}{}{1}{n}\genfrac{(}{)}{0pt}{}{nk}{m}\genfrac{(}{)}{0pt}{}{n}{n-m-1}
\end{align}
\end{cor}
\begin{proof}
By the definition of $k$-Schr\"oder paths of type $B$, one can associate 
$2^{m+1}$ $k$-Schr\"oder paths of type $B$ to a $k$-Dyck path with $m+1$ peaks of size $n$
by replacing the pattern $UD$ by $H$.
By comparing the above property with Eq. (\ref{eq:SchBnk}), 
we obtain the desired expression.
\end{proof}

The integer sequences given by Eq. (\ref{eq:kDyckmpeaks}) appear in OEIS \cite{Slo} as 
A001263, A108767, A173020, and A173621 for $k=1,2,3$ and $4$.
The sequence for $k=1$ is Narayana numbers and the one for $k\ge2$
is generalized Runyon numbers studied in \cite{Cig87,Son05}.

\subsection{Generalized recurrence relation with distance two}
\label{sec:Gdist2}
The recurrence relation for the dimer model with distance more than two sites 
is transformed to the following recurrence relation.
\begin{defn}
Let $\zeta(m,r,s)$ be a formal power series
\begin{align}
\label{eq:zetam}
\begin{split}
1&=(1+rs(m-2))\zeta(r)+rs\zeta(r)\zeta(rs)+ rs(1+rs^2(m-1))\zeta(r)\zeta(rs)\zeta(rs^2) \\
&\quad+r^2s^3\zeta(r)\zeta(rs)\zeta(rs^2)\zeta(rs^3).
\end{split}
\end{align}
\end{defn}

Below, we consider a special case $m=1$. 
To describe the formal power series for $m\ge2$ in terms of generalized lattice 
paths remains an open problem (see Section \ref{sec:OP}).

Given a $2$-Dyck path $\mu$ in $\mathtt{Dyck}(n;2)$, we define a statistics 
$\mathrm{Area}(\mu)$ as the number of unit boxes above the lowest path $\mu_{0}$ 
and below the path $\mu$.

The formal power series $\zeta(m=1,r,s)$ can be expressed in terms of $k$-Dyck paths.
\begin{prop}
The formal power series $\zeta(m=1,r,s)$ satisfies
\begin{align}
\label{eq:zetam1}
\zeta(m=1,r,s)=\sum_{0\le n}(-rs)^{n}
\sum_{\mu\in\mathtt{Dyck}(n;2)}
s^{\mathrm{Area}(\mu)}.
\end{align}
\end{prop}
\begin{proof}
Suppose that a formal power series $\zeta'(r,s)$ satisfies 
\begin{align}
\label{eq:zetadash}
1=\zeta'(r,s)+rs\zeta'(r,s)\zeta'(rs,s)\zeta'(rs^2,s).
\end{align}
Then, $\zeta'(r,s)$ satisfies also the recurrence equation (\ref{eq:zetam}) 
with $m=1$.
Therefore, it is enough to show that the formal power series defined by Eq. (\ref{eq:zetam1}) 
satisfies the recurrence equation (\ref{eq:zetadash}).

Suppose that $\zeta'(r,s)$ is the formal power series defined by Eq. (\ref{eq:zetam1}). 
Recall that we have a canonical bijection between a $2$-Dyck path 
and a ternary tree.
A node in a ternary tree has three children. 
Each children has also a ternary tree below it.

To be compatible with the statistics $\mathrm{Area}(\mu)$, 
we have to shift $r$ to $rs^{j}$ with $j\in\{0,1,2\}$ 
corresponding to the position of the child edges.
Thus, it is straightforward to observe that $\zeta'(r,s)$ satisfies
the recurrence equation (\ref{eq:zetadash}).
This completes the proof.
\end{proof}

\subsection{Type I}
\label{sec:glpI}
We study a formal power series associated to the generating function of the dimer model
of type I.
The formal power series can be expressed in terms of $k$-Dyck paths, $k$-Schr\"oder 
paths of type $A$.
We also show that, under a certain specialization of formal variables, the formal 
power series can be expressed in terms of $k$-Motzkin paths.

\subsubsection{\texorpdfstring{$k$}{k}-Dyck paths}

We generalize the recurrence relation in such a way that 
it reflects combinatorial structures.
We replace the role of Dyck paths by that of $k$-Dyck paths.
From Definition \ref{defn:rrtypeI}, we consider the 
following difference equation for a formal power series 
which we call type I.
\begin{defn}
Let $\nu(m,r,s,u)$ be a formal power series satisfying 
\begin{align}
\label{eq:rrnu}
f(rs)\nu(m,r,s,u)=1-g(rs)\prod_{j=0}^{k}\nu(m,rs^{j},s,u),
\end{align}
where $f(x)$ and $g(x)$ are defined in Definition \ref{defn:fg}.
\end{defn}
Note that if we set $k=1$, Eq. (\ref{eq:rrnu}) coincides with 
the recurrence relation (\ref{eq:defchi}).	

Let $\mu$ be a $k$-Dyck path of size $n$.
Let $\mathrm{Peak}(k;\mu)$ be the number of pattern $UD^{k}$ in $\mu$.
If we set $k=1$, then $\mathrm{Peak}(1;\mu)$ is equal to $\mathrm{Peak}(\mu)$.

\begin{defn}
\label{defn:wtI}
We define the weight $\mathrm{wt}^{I}(k;\mu)$ for a $k$-Dyck path $\mu$ as 
\begin{align*}
\mathrm{wt}^{I}(k;\mu):=(m'+1)^{\mathrm{Peak}(k;\mu)}(u'+1)^{\mathrm{Peak}(k;\mu)}
(m'u'+u'+1)^{n-\mathrm{Peak}(k;\mu)}.
\end{align*}
\end{defn}

\begin{prop}
The formal power series $\nu(m,r,s,u)$ is given by
\begin{align*}
\nu(m,r,s,u)=\sum_{0\le n}(-rs)^{n}\sum_{\mu\in\mathtt{Dyck}(n;k)}s^{|Y(\mu_{0}/\mu)|}\mathrm{wt}^{I}(k;\mu).
\end{align*}
\end{prop}
\begin{proof}
Let 
\begin{align*}
\nu_{n}(s):=\sum_{\mu\in\mathtt{Dyck}(n;k)}s^{|Y(\mu_{0}/\mu)|}\mathrm{wt}^{I}(k;\mu).
\end{align*}
Let $J(a;S)$, $a\le k+1$, be the set 
\begin{align*}
J(a;S):=\left\{
(j_a,j_{a+1},\ldots,j_{k+1})\Big|
\sum_{p=a}^{k+1}j_{p}=S
\right\}.
\end{align*}
We denote $\mathbf{j}:=(j_a,j_{a+1},\ldots,j_{k+1})$. 
From Eq. (\ref{eq:rrnu}), $\nu_{n}(s)$ should satisfy
\begin{align}
\label{eq:rrnu2}
\begin{split}
\nu_{n}(s)&=m'\nu_{n-1}(s)+(m'u'+u'+1)\sum_{\mathbf{j}\in J(1;n-1)}\prod_{p=1}^{k+1}\nu_{j_p}(s)s^{(p-1)j_{p}}, \\
&=(m'+1)(u'+1)\nu_{n-1}(s)+(m'u'+u'+1)\sum_{j_1=0}^{n-2}\nu_{j_1}(s)\sum_{\mathbf{j}\in J(2;n+1-j_{1})}\prod_{p=2}^{k+1}\nu_{j_p}(s)s^{(p-1)j_{p}}, \\
&=\nu_{1}(s)\nu_{n-1}(s)+(m'u'+u'+1)\sum_{j_1=0}^{n-2}\nu_{j_1}(s)\sum_{\mathbf{j}\in J(2;n+1-j_{1})}\prod_{p=2}^{k+1}\nu_{j_p}(s)s^{(p-1)j_{p}}.
\end{split}
\end{align}
Any $k$-Dyck path can be written as a concatenation of $k$-Dyck path, one of which is prime.
A prime $k$-Dyck path can be constructed from the prime path $UD^{k}$ of size $1$.
We insert $k$ $k$-Dyck paths between the first $U$ and the last $D$.
If we insert $k$-Dyck path of size $j_{p}$ at $p$-th position from right in $UD^{k}$, 
the area below the prime path is increased by $(p-1)j_{p}$.
Therefore, $\nu_{n}(s)$ satisfies Eq. (\ref{eq:rrnu2}), which completes the proof.
\end{proof}

We expand $\nu(m,r,s=1,u)$ as 
\begin{align*}
\nu(m,r,s=1,u)=:\sum_{0\le n}(-r)^{n}S_{k}(n;m',u').
\end{align*}
As we will see below, the number $S_{k}(n;m',u')$ with a specialization of $(m',u')$
is equal to the number of generalized lattice path with a certain weight.
Especially, we are interested in the specialization $m',u'\in\{0,1\}$.

\begin{prop}
\label{prop:typeIsp}
The number $S_{k}(n;m',u')$ with $m',u'\in\{0,1\}$ enumerates the 
following combinatorial objects.
\begin{enumerate}
\item The number $S_{k}(n;0,0)$ is given by $|\mathtt{Dyck}(n;k)|$, which 
is the Fuss--Catalan number.
\item The number $S_{k}(n;0,1)$ is given by $2^{n}|\mathtt{Dyck}(n;k)|$. 
Namely, the number of $k$-Dyck paths with two colors for each up step.
\item The number $S_{k}(n;1,0)$ is the number of $k$-Schr\"oder paths of type A, 
that is, $|\mathtt{Sch}^{A}(n;k)|$. 
\item The number $S_{k}(n;1,1)$ is the number of weighted $k$-Dyck paths $\mu$ 
such that an up step having the pattern $UD^{k}$ has weight $4$ and other 
up steps have weight $3$.
\end{enumerate}
\end{prop}
\begin{proof}
(1). By Definition \ref{defn:wtI} of the weight given to a $k$-Dyck path,
it is obvious that $S_{k}(n;0,0)$ enumerates $k$-Dyck paths of size $n$.

(2). We have $\mathrm{wt}^{I}(k;\mu)=2^{n}$ if we set $(m',u')=(0,1)$.
Since $n$ is the number of up steps in $\mu$, $S_{k}(n;0,1)$ is 
equal to $2^{n}|\mathtt{Dyck}(n;k)|$.

(3). We have $\mathrm{wt}^{I}(k;\mu)=2^{\mathrm{Peak}(k;\mu)}$ if we 
set $(m',u')=(1,0)$.
Since a peak is a pattern $UD^{k}$, we may replace this pattern by 
a horizontal step $(k,1)$.
Further, even if we replace it by a horizontal step, the newly obtained 
path never goes below $y=x/k$.
This implies that the new path is a $k$-Schr\"oder path of type A of size $n$.
From these observations, the number $S_{k}(n;1,0)$ is given 
by $|\mathtt{Sch}(n;k)|$.

(4). Since $(m',u')=(1,1)$, the weight is 
$\mathrm{wt}^{I}(k;\mu)=4^{\mathrm{Peak}(k;\mu)}3^{n-\mathrm{Peak}(k;\mu)}$.
Thus the statement in (4) is a direct consequence of this observation.
\end{proof}

The next proposition gives an explicit expression for the number $S_{k}(n;1,1)$.
\begin{prop}
The number $S_{k}(n;1,1)$ is given by 
\begin{align*}
S_{k}(n;1,1)=
\genfrac{}{}{}{}{1}{n}\sum_{j=0}^{n}3^{j}\genfrac{(}{)}{0pt}{}{n}{j}\genfrac{(}{)}{0pt}{}{n+kj}{n-1}.
\end{align*}
\end{prop}
\begin{proof}
Let $\nu(r):=\nu(m,r,s=1,u)$.  
Since $(m',u')=(1,1)$, we have 
\begin{align*}
\nu(r)=1+r H(\nu(r)),
\end{align*}
where $H(x)=x+3x^{k+1}$.
By applying the Lagrange inversion theorem  for $H(x)=x+3x^{k+1}$, we obtain  
\begin{align*}
[r^{n}]\nu(r)=\genfrac{}{}{}{}{1}{n}[\lambda^{n-1}]H(1+\lambda)^{n}.
\end{align*}
We have
\begin{align*}
[\lambda^{n-1}]H(1+\lambda)^{n}&=[\lambda^{n-1}]\left((1+\lambda)+3(1+\lambda)^{k+1})\right)^{n}, \\
&=[\lambda^{n-1}]\left(\sum_{j=0}^{n}\genfrac{(}{)}{0pt}{}{n}{j}3^{j}(1+\lambda)^{n+kj}\right), \\
&=[\lambda^{n-1}]\left(\sum_{j=0}^{n}\sum_{l=0}^{n+kj}
\genfrac{(}{)}{0pt}{}{n}{j}3^{j}\genfrac{(}{)}{0pt}{}{n+kj}{l}\lambda^{l}\right), \\
&=\sum_{j=0}^{n}3^{j}\genfrac{(}{)}{0pt}{}{n}{j}\genfrac{(}{)}{0pt}{}{n+kj}{n-1},
\end{align*}
which completes the proof.
\end{proof}

From (3) in Proposition \ref{prop:typeIsp}, 
we compute the number of $k$-Schr\"oder paths of type $A$ of size $n$ by use of 
the Lagrange inversion theorem.
\begin{prop}
\label{prop:kSchA}
The number of $k$-Schr\"oder paths of type A of size $n$ is given by
\begin{align*}
|\mathtt{Sch}^{A}(n;k)|
=\genfrac{}{}{}{}{1}{n}\sum_{j=0}^{n}\genfrac{(}{)}{0pt}{}{n}{j}\genfrac{(}{)}{0pt}{}{n+kj}{n-1}.
\end{align*}
\end{prop}
\begin{proof}
Under the specialization $(m',u')=(1,0)$, the formal power series $\nu(r):=\nu(m=2,-r,s=1,u=1)$
satisfies the relation $\nu(r)=1-r(\nu(r)+\nu(r)^{k+1})$.
By applying Lagrange inversion theorem, we have 
\begin{align*}
[r^{n}]\nu(r)&=\genfrac{}{}{}{}{1}{n}[\lambda^{n-1}]\left((1+\lambda)+(1+\lambda)^{k+1}\right)^{n}, \\
&=\genfrac{}{}{}{}{1}{n}\sum_{j=0}^{n}\genfrac{(}{)}{0pt}{}{n}{j}\genfrac{(}{)}{0pt}{}{n+kj}{n-1},
\end{align*}
which completes the proof.
\end{proof}

The integer sequences $|\mathtt{Sch}^{A}(n;k)|$ appear in OEIS \cite{Slo} as 
A006318, A346626, A349310 and A349311 for $k=1,2,3$ and $4$ respectively.

By comparing Proposition \ref{prop:SchB} with Proposition \ref{prop:kSchA},
it is obvious that $\mathtt{Sch}^{A}(n;k)$ and $\mathtt{Sch}^{B}(n;k)$
coincide with each other when $k=1$.
By setting $k=1$, we recover the expression (\ref{eq:Sch}).

\subsubsection{Specialization and generalized Motzkin paths}
We consider the recurrence relation of type I with a specialization 
\begin{align*}
m=2, \qquad u=\genfrac{}{}{}{}{1}{2}(1+r^ks^{k(k+3)/2}).
\end{align*}
By substituting this specialization into Eq. (\ref{eq:rrnu}), 
the formal power series $\nu(r)$ satisfies 
\begin{align}
\label{eq:IMot}
\nu(r)=1-rs \nu(r)-r^{k+1}s^{(k+1)(k+2)/2}\prod_{j=0}^{k}\nu(rs^j).
\end{align}
Since we specialize $m$ and $u$, $\nu(r)$ is a formal power 
series in $r$ and $s$.

Let $\mu$ be a $k$-Motzkin path of size $n$.
The lowest path $\mu_{0}$ is given by $H^{n}$.
Recall that $|Y(\mu_0/\mu)|$ is the number of unit boxes in the skew Young 
diagram $Y(\mu_0/\mu)$ surrounded by two paths $\mu_0$ and $\mu$.
\begin{prop}
The generating function $\nu(r)$ can be expressed in terms of 
$k$-Motzkin paths. 
\begin{align}
\label{eq:nuMot}
\nu(r)=\sum_{0\le n}(-rs)^{n}\sum_{\mu\in\mathtt{Mot}(n;k)}s^{|Y(\mu_{0}/\mu)|}.
\end{align}
\end{prop}
\begin{proof}
Let $\nu'(r):=\nu(-r)$ by use of Eq. (\ref{eq:nuMot}). 
We consider $k$-Motzkin paths of length $n$.
The area below the path $UD^{k}$ is $k(k+1)/2$ by a simple calculation.
Since the size of $UD^{k}$ is $k+1$, $UD^{k}$ has the weight $s^{(k+1)(k+2)/2}$.
Then, the generating function $\nu'(r)$ satisfies
\begin{align}
\label{eq:IMot2}
\nu'(r)=1+rs\nu'(r)+r^{k+1}s^{(k+1)(k+2)/2}\prod_{j=0}^{k}\nu'(rs^{j}).
\end{align}
The first term $1$ corresponds to the empty path.
The term $rs\nu'(r)$ corresponds to $k$-Motzkin paths starting from $H$.
The third term comes from $k$-Motzkin paths starting from $U$.
From these observations, $\nu(r)=\nu'(-r)$ satisfies the 
recurrence relation (\ref{eq:IMot}), 
which completes the proof.
\end{proof}

Further specialization $s=1$ implies that $\nu(r)$ is the ordinary generating 
function of $|\mathtt{Mot}(n,k)|$.
Then, we enumerate the $k$-Motzkin paths in $\mathtt{Mot}(n,k)$ by use of 
the functional relation for $\nu(r)$.
\begin{prop}
\label{prop:kMot}
The total number of $k$-Motzkin paths of size $n$ is given by
\begin{align*}
|\mathtt{Mot}(n,k)|=\genfrac{}{}{}{}{1}{n+1}\sum_{j=0}^{\lfloor n/k\rfloor}
\genfrac{(}{)}{0pt}{}{n+1}{n-kj}\genfrac{(}{)}{0pt}{}{n-kj}{j}.
\end{align*}
\end{prop}
\begin{proof}
We specialize $s=1$ in Eq. (\ref{eq:IMot2}), and we obtain 
\begin{align*}
r=\genfrac{}{}{}{}{\widetilde{\nu}(r)}{1+\widetilde{\nu}(r)+\widetilde{\nu}(r)^{k+1}},
\end{align*}
where $\widetilde{\nu}(r)=r\nu(r)$.
By Lagrange inversion theorem, we have 
\begin{align*}
|\mathtt{Mot}(n,k)|&=[r^{n}]\nu'(r), \\
&=\genfrac{}{}{}{}{1}{n+1}[x^{n}](1+x+x^{k+1})^{n+1}, \\
&=\genfrac{}{}{}{}{1}{n+1}[x^{n}]\left(\sum_{p=0}^{n+1}\sum_{q=0}^{p}x^{p+kq}
\genfrac{(}{)}{0pt}{}{n}{p}\genfrac{(}{)}{0pt}{}{p}{q}\right), \\
&=\genfrac{}{}{}{}{1}{n+1}\sum_{j=0}^{\lfloor n/k\rfloor}
\genfrac{(}{)}{0pt}{}{n+1}{n-kj}\genfrac{(}{)}{0pt}{}{n-kj}{j},
\end{align*} 
which completes the proof.
\end{proof}

The integer sequences $|\mathtt{Mot}(n;k)|$ appear in OEIS \cite{Slo}
as A001006, A071879 and A127902 for $k=1,2$ and $3$ respectively.

\subsection{Type II}
From Definition \ref{defn:GnII}, we consider 
the following difference equation for type II recurrence relation.

Fix an integer $k\ge2$.
\begin{defn}
\label{defn:xi}
Let $\xi(m,r,s,u)$ be a formal power series satisfying 
\begin{align*}
f(rs)\xi(m,r,s,u)=1+rs(m-1)\prod_{j=0}^{k-1}\xi(m,rs^{j},s,u)-mrsu\prod_{j=0}^{k}\xi(m,rs^{j},s,u).
\end{align*}
\end{defn}

Let $\mu$ be a $k$-Dyck path of size $n$.
We call a pattern $DU$ a valley.
If $\mu$ has a pattern $D^{p}U$ with maximal $p$ at a valley $v$,
we define the size of $v$ as $|v|=p$.
We denote by $\mathcal{V}(\mu)$ the set of valleys $v$ in $\mu$.
Then, we define 
\begin{align*}
\mathrm{Val}^{II}(k;\mu):=\sum_{v\in\mathcal{V}(\mu)}\delta(0<|v|\le k-1),
\end{align*}
where $\delta(P)$ is the Kronecker delta function, that is, $\delta(P)=1$ if 
the statement $P$ is true and $\delta(P)=0$ otherwise. 

\begin{defn}
\label{defn:xiwt}
We define the weight $\mathrm{wt}^{II}(k;\mu)$ for a $k$-Dyck path $\mu$ by
\begin{align*}
\mathrm{wt}^{II}(k;\mu):=(m'+1)^{n-\mathrm{Val}^{II}(k;\mu)}(u'+1)^{n-\mathrm{Val}^{II}(k;\mu)}
(m'u'+u'+1)^{\mathrm{Val}^{II}(k;\mu)}.	
\end{align*}
\end{defn}

\begin{prop}
\label{prop:xi}
The formal power series $\xi(m,r,s,u)$ is given by 
\begin{align}
\label{eq:expxi}
\xi(m,r,s,u)=\sum_{0\le n}(-rs)^{n}
\sum_{\mu\in\mathtt{Dyck}(n;k)}s^{|Y(\mu_{0}/\mu)|}\mathrm{wt}^{II}(k;\mu).
\end{align}
\end{prop}
\begin{proof}
We show that Eq. (\ref{eq:expxi}) satisfies the 
recurrence relation in Definition \ref{defn:xi}.

Let $F_{1}(r):=F_{1}(m,r,s,u)$ be the generating function of $k$-Dyck 
paths of size greater than zero with weight defined in Definition \ref{defn:xiwt} such that 
the first valley is of size greater than or equal to $k$, or it has no valleys.
We denote by $S_{1}$ the set of such $k$-Dyck paths.
Similarly, $F_{2}(r):=F_{2}(m,r,s,u)$ be the generating function of $k$-Dyck paths
such that the size of the first valley is less than $k$.
We denote by $S_{2}$ the set of such $k$-Dyck paths.
Note that a $k$-Dyck path in $F_{2}(r)$ has at least one valley, which implies 
that the length of paths are greater than one.
Let $\xi(r):=\xi(m,-r,s,u)$.
Then, we have the following system of functional equations:
\begin{align}
\label{eq:xirec}
\xi(r)&=1+F_{1}(r)+F_{2}(r), \\
\label{eq:F1rec}
F_{1}(r)&=rsf_4\xi(r)+rsf_{4}F_{1}(rs^{k})\prod_{j=0}^{k-1}\xi(rs^{j}), \\
\label{eq:F2rec}
F_{2}(r)&=rsf_{3}\left(\prod_{j=1}^{k-1}\xi(rs^{j})-1\right)\xi(r)+rsf_{4}F_{2}(rs^k)\prod_{j=0}^{k-1}\xi(rs^{j}),
\end{align}
where 
\begin{align}
\label{eq:f3f4}
f_{3}=(m'u'+u'+1), \qquad f_{4}=(m'+1)(u'+1).
\end{align}
Below, we briefly explain Eqs. (\ref{eq:xirec}) to (\ref{eq:F2rec}). 
The first equation (\ref{eq:xirec}) is obvious from the definitions of $F_{1}(r)$ and $F_{2}(r)$.
The constant term $1$ corresponds to the empty path.	

Since the length of a $k$-Dyck path $\mu$ in $S_{1}$ is not zero, $\mu$ has the following 
form:
\begin{align}
\label{eq:ins}
U\circ \mu_1 \circ D\circ \mu_2 \circ D\circ \ldots \circ \mu_{k} \circ D\circ \mu_{k+1},
\end{align}
where $\mu_{i}$, $1\le i\le k+1$, is a $k$-Dyck path.
The first term in Eq. (\ref{eq:F1rec}) corresponds to the paths with $\mu_{i}=\emptyset$, $1\le i\le k$.
To have the contribution in $F_{1}(r)$ requires that $\mu_1$ is also in $S_{1}$.
Note that the weight involves the area of a $k$-Dyck path. 
This is realized by $r\rightarrow rs^{k+1-j}$ for the $j$-th path $\mu_{j}$.
By combining these observations, we have the second term in Eq. (\ref{eq:F1rec}).

Similarly, we obtain Eq. (\ref{eq:F2rec}). 
Since the length of a path in $S_2$ is larger than one, at least one $\mu_{i}$ is 
not empty for $1\le i\le k$. 
This corresponds to $\left(\prod_{j=1}^{k-1}\xi(rs^{j})-1\right)$.
The second term corresponds to the path $\mu_1\in S_{2}$.
Note that any $k$-Dyck path has at least $k$ down steps at the last end, 
the path in Eq. (\ref{eq:ins}) with $\mu_1\in S_{2}$ has a factor $f_4$ since 
all newly obtained valleys have the size greater than $k$.

From Eqs. (\ref{eq:xirec}) to (\ref{eq:F2rec}), we have 
\begin{align*}
\xi(r)-1&=F_{1}(r)+F_{2}(r), \\
&=rsm'\xi(r)+rsf_3\prod_{j=0}^{k-1}\xi(rs^{j})+rsf_4\left(\xi(rs^{k})-1\right)\prod_{j=1}^{k-1}\xi(rs^{j}), \\
&=rsm'\xi(r)-m'rs\prod_{j=0}^{k-1}\xi(rs^{j})+rsf_4\prod_{j=1}^{k}\xi(rs^{j}),
\end{align*}
where we have used $f_{4}-f_{3}=m'$.
Since $\xi(-r)$ satisfies the recurrence relation in Definition \ref{defn:xi}, 
this completes the proof.
\end{proof}

\begin{prop}
\label{prop:II22}
Suppose that $(m,r,s,u)=(2,r,1,2)$ and $\xi(r):=\xi(2,r,1,2)$.
Then, we have 
\begin{align}
\label{eq:IIxin}
[r^{n}]\xi(-r)=
\genfrac{}{}{}{}{1}{n}
\sum_{p=0}^{n}\sum_{q=0}^{n-p}3^{n-p-q}4^{q}\genfrac{(}{)}{0pt}{}{n}{p}\genfrac{(}{)}{0pt}{}{n-p}{q}
\genfrac{(}{)}{0pt}{}{nk-p(k-1)}{n-1-q}.
\end{align}
\end{prop}
Let $N:=\mathrm{Val}^{II}(k;\mu)$ for a $k$-Dyck path $\mu$. 
Then, the proposition above implies that 
the weighted sum of $k$-Dyck paths $\mu$ of size $n$ such that $\mu$ has 
the weight $3^{N}4^{n-N}$ is given by Eq. (\ref{eq:IIxin}).
\begin{proof}[Proof of Proposition \ref{prop:II22}]
If we set $(m,s,u)=(2,1,2)$, we have 
\begin{align*}
\xi(r)=1+r\widetilde{H}(\xi(r)),
\end{align*}
where $\widetilde{H}(x)=-x+x^{k}-4x^{k+1}$.
By Lagrange inversion theorem,  we have 
\begin{align*}
[r^{n}]\xi(r)&=\genfrac{}{}{}{}{(-1)^{n}}{n}[x^{n-1}]\left(\widetilde{H}(1+x)\right)^{n}, \\
&=\genfrac{}{}{}{}{(-1)^{n}}{n}[x^{n-1}]\left((1+x)+(1+x)^{k}(4x+3)\right)^{n}, \\
&=\genfrac{}{}{}{}{(-1)^{n}}{n}[x^{n-1}]\sum_{p=0}^{n}\genfrac{(}{)}{0pt}{}{n}{p}
(1+x)^{nk-p(k-1)}(4x+3)^{n-p}, \\
&=\genfrac{}{}{}{}{(-1)^{n}}{n}[x^{n-1}]\sum_{p=0}^{n}\sum_{q=0}^{nk-p(k-1)}\sum_{r=0}^{n-p}
4^{r}3^{n-p-r}\genfrac{(}{)}{0pt}{}{n}{p}\genfrac{(}{)}{0pt}{}{nk-p(k-1)}{q}\genfrac{(}{)}{0pt}{}{n-p}{r}x^{q+r}, \\
&=\genfrac{}{}{}{}{(-1)^{n}}{n}
\sum_{p=0}^{n}\sum_{q=0}^{n-p}3^{n-p-q}4^{q}\genfrac{(}{)}{0pt}{}{n}{p}\genfrac{(}{)}{0pt}{}{n-p}{q}
\genfrac{(}{)}{0pt}{}{nk-p(k-1)}{n-1-q},
\end{align*}
which completes the proof.
\end{proof}
From Proposition \ref{prop:II22}, we have $\xi(r)=1+4r\xi(r)^2$ when $k=1$. 
By Lagrange inversion theorem,  
\begin{align*}
[r^{n}]\xi(-r)=\genfrac{}{}{}{}{4^n}{n+1}\genfrac{(}{)}{0pt}{}{2n}{n},
\end{align*}
for $k=1$. 
These numbers are a product of $4^n$ and the $n$-th Catalan number, 
and appear as a sequence A052704 (or equivalently A151403) in OEIS \cite{Slo}.

Below, we consider a formal power series $\alpha(r)$ associated to 
the recurrence relation given in Definition \ref{defn:HII}.
The formal power series $\alpha(r)$ satisfies 
\begin{align}
\label{eq:alpha}
1=f(rs)\alpha(r)-rs(m-1)\alpha(r)\alpha(rs)+mrsu\prod_{j=0}^{k}\alpha(rs^{j}).
\end{align}
The specialization $k=1$ with $(m,s,u)=(1,1,1)$ gives the formal power series 
for the Dyck paths.
Similarly, the specialization $k=1$ with $(m,s,u)=(2,1,1)$ gives the 
formal power series for the Schr\"oder paths.

We consider the specialization $(m,r,s,u)=(2,-r,1,1)$ of the relation (\ref{eq:alpha}).
The $\alpha(r)$ satisfies 
\begin{align*}
\alpha(r)=1+r(\alpha(r)-\alpha(r)^{2}+2\alpha(r)^{k+1}).
\end{align*}
By applying the Lagrange inversion theorem, we have 
\begin{align}
\label{eq:alphaSch}
\begin{split}
[r^{n}]\alpha(r)&=
\genfrac{}{}{}{}{1}{n}[y^{n-1}]\left(-y(1+y)+2(1+y)^{k+1}\right)^{n}, \\
&=\genfrac{}{}{}{}{1}{n}\sum_{j=0}^{n}(-1)^{j}2^{n-j}
\genfrac{(}{)}{0pt}{}{n}{j}\genfrac{(}{)}{0pt}{}{k(n-j)+n}{n-1-j}.
\end{split}
\end{align}
The integer sequence (\ref{eq:alphaSch}) for $k=1$ appears in OEIS \cite{Slo}
as the large Schr\"oder numbers A006318.
First few values of the sequences are 
\begin{align*}
&2, 10, 70, 566, 4970, 46050, 443134, 4385790, \ldots \quad(k=2),\\
&2, 14, 142, 1674, 21498, 291814, 4118006, 59808210,\ldots \quad (k=3), \\
&2, 18, 238, 3670, 61754, 1099322, 20356118, 388071310,\ldots \quad (k=4).
\end{align*}

Define 
\begin{align*}
A(n,k,j):=\genfrac{}{}{}{}{1}{n}\genfrac{(}{)}{0pt}{}{n}{j}\genfrac{(}{)}{0pt}{}{k(n-j)+n}{n-1-j}.
\end{align*}
We introduce combinatorial lattice paths whose total number is given by $A(n,k,j)$.

Let $\mathcal{A}(n,k,j)$ be the set of lattice paths satisfying the following four conditions: 
\begin{enumerate}
\item A path from $(0,0)$ to $((k+1)n,0)$ which is in the first quadrant.
\item A path consists of three steps $u=(k,1)$, $U=(1,k)$ and $d=(1,-1)$.
\item There is no peaks of the form $ud$.
\item The number of $u$ in a path is exactly $j$.
\end{enumerate}

\begin{prop}
We have 
\begin{align*}
|\mathcal{A}(n,k,j)|=A(n,k,j).
\end{align*}
\end{prop}
\begin{proof}
Let $\beta(r)$ be a formal power series defined by
\begin{align*}
\beta(r):=\sum_{i\ge0}\sum_{j=0}^{n}r^{i}z^{j}A(n,k,j).
\end{align*}
By applying the Lagrange inversion theorem to $A(n,k,j)$,
$\beta(r)$ should satisfy the following recurrence relation:
\begin{align}
\label{eq:beta}
\beta(r)=1+r(z(\beta(r)-1)\beta(r)+\beta(r)^{k+1}).
\end{align}
Recall $\mathcal{A}(n,k,j)$ consists of lattice paths satisfying 
the conditions (1) to (4).
A path in $\mathcal{A}(n,k,0)$ consists of $k$-Dyck paths 
consisting of $n$ $U$'s and $nk$ $d$'s.
From (2), we have $Ud^{k-1}=u=(k,1)$.
From (3), the step $u$ is followed by $U$ or $u$.
For $n=1$, we have a unique path: $Ud^{k}$. 
Note that the path $ud$ is eliminated by the condition (2).
For general $n\ge2$, we have two ways to obtain the lattice paths in $\mathcal{A}(n,k,j)$.
First, we insert $k+1$ paths into $Ud^{k}$ between $i$-th and $i+1$-th steps with $1\le i\le k-1$ and 
the right to the $k$-th step. 
Secondly, we insert two paths into $ud$ such that we insert a path of length $n\ge1$ between 
$u$ and $d$ and a path of length $n\ge0$ right to $d$.
This is because the pattern $ud$ is not admissible and we have to insert a path of length $n\ge1$
between $u$ and $d$.
Since the path $Ud^{k}$ (resp. $ud$) gives the factor $r$ (resp. $rz$),
the formal power series $\beta(r)$ satisfies Eq. (\ref{eq:beta}), which 
completes the proof. 
\end{proof}
For $k=1$, the integer sequence $A(n,1,j)$ appears in OEIS \cite{Slo} as A033282.
First few values of $A(n,k,j)$ are 
\begin{align*}
&\{1\}, \{3, 1\}, \{12, 7, 1\}, \{55, 45, 12, 1\}, \{273, 286, 110, 18, 1\}, \ldots \quad (k=2), \\
&\{1\}, \{4, 1\}, \{22, 9, 1\}, \{140, 78, 15, 1\}, \{969, 680, 182, 22, 1\}, \ldots \quad (k=3), \\
&\{1\}, \{5, 1\}, \{35, 11, 1\}, \{285, 120, 18, 1\}, \{2530, 1330, 272, 26, 1\}, \ldots \quad (k=4).
\end{align*}
For example, $\{12,7,1\}$ for $k=2$ means that $A(3,2,0)=12$, $A(3,2,1)=7$ and $A(3,2,2)=1$.

\subsection{Type III}
Fix an integer $k\ge2$.
We define a formal power series $\kappa(r):=\kappa(m,r,s,u)$ as follows.
\begin{defn}
\label{defn:kappa}
The formal power series $\kappa(r)$ satisfies 
the recurrence relation
\begin{align*}
\kappa(r)=1-rs(m-1)\prod_{j=0}^{k-1}\kappa(rs^{j})-rs(m(u-1)+1)\prod_{j=0}^{k}\kappa(rs^{j}).
\end{align*}
\end{defn}

The formal power series $\kappa(r)$ has an expression in terms of $k$-Dyck paths.
Given a $k$-Dyck path $\mu$, let $\mathrm{Peak}(\mu)$ be the number of 
peaks in $\mu$, {\it i.e.}, the number of the pattern $UD$ in $\mu$.
Then, we define the weight $\mathrm{wt}^{III}(\mu)$ for $\mu$ by
\begin{align}
\label{eq:wtIII}
\mathrm{wt}^{III}(\mu):=
(m'+1)^{\mathrm{Peak}(\mu)}(u'+1)^{\mathrm{Peak}(\mu)}(m'u'+u'+1)^{n-\mathrm{Peak}(\mu)}.
\end{align}

\begin{remark}
Note the definition of the weight (\ref{eq:wtIII}) is similar to the one fro type I (\ref{defn:wtI}).
The difference between type I and type III is that not only the pattern $UD^k$ but also the pattern $UD^{p}$ 
with $1\le p\le k-1$  are  counted as a peak in type III.
Given $\mu\in\mathtt{Dyck}(n;k)$, the difference $\mathrm{Peak}(\mu)-\mathrm{Peak}(k;\mu)$ 
is always non-negative.	
\end{remark}

\begin{prop}
The formal power series $\kappa(r)$ is given by
\begin{align}
\label{eq:nu}
\kappa(r)=\sum_{0\le n}(-rs)^{n}
\sum_{\mu\in\mathtt{Dyck}(n;k)}
s^{|Y(\mu_{0}/\mu)|}\mathrm{wt}^{III}(\mu).
\end{align}
\end{prop}
\begin{proof}
Let $\kappa'(r)$ be the generating function $\kappa'(r):=\kappa(-r)$ given by Eq. (\ref{eq:nu}).
We prove the proposition in a similar manner to the proof of Proposition \ref{prop:xi}.
Then, the generating function $\kappa'(r)$ satisfies 
\begin{align*}
\kappa'(r)&=1+rsf_{4}\prod_{j=0}^{k-1}\kappa'(rs^{j})+rsf_3(\kappa'(rs^{k})-1)\prod_{j=0}^{k-1}\kappa'(rs^{j}), \\
&=1+rsm'\prod_{j=0}^{k-1}\kappa'(rs^{j})+rsf_3\prod_{j=0}^{k}\kappa'(rs^{j}),
\end{align*}
where $f_3$ and $f_4$ are defined in Eq. (\ref{eq:f3f4}).
From this recurrence relation, it is obvious that $\kappa(r)$ satisfies the recurrence relation 
of type III.
\end{proof}

We define
\begin{align*} 
S^{III}(n; m',u'):=(-1)^{n}[r^{n}]\kappa(m',r,s=1,u').
\end{align*}
\begin{prop}
\label{prop:SIII}
The number $S^{III}(n;m',u')$ with $m',u'\in\{0,1\}$ enumerates the 
following generalized lattice paths.
\begin{enumerate}
\item $S^{III}(n;0,0)$ is the number of $k$-Dyck paths of size $n$.
\item $S^{III}(n;0,1)$ is the number of $k$-Dyck paths of size $n$ with two types of up steps.
Namely, $S^{III}(n;0,1)=2^{n}|\mathtt{Dyck}(n;k)|$.
\item $S^{III}(n;1,0)$ is the number of paths from $(0,0)$ to $((k+1)n,0)$ which stay 
in the first quadrant and where each step is $(1,k)$, $(2,k-1)$ or $(1,-1)$.
\item $S^{III}(n;1,1)$ is the weighted sum of $k$-Dyck paths $\mu$ of size $n$ where 
the path $\mu$ has a weight $4^{\mathrm{Peak}(\mu)}3^{n-\mathrm{Peak}(\mu)}$.
\end{enumerate}
\end{prop}
\begin{proof}
(1), (2) and (4) are obvious from the definition of the weight given to a path $\mu$.

(3). We have an obvious bijection between a $k$-Dyck path and a lattice path from $(0,0)$ to $((k+1)n,0)$
such that it stays in the first quadrant and each step is either $(1,k)$ or $(1,-1)$.
The specialization $(m',u')=(1,0)$ gives the weight $2^{\mathrm{Peak}(\mu)}$ to a $k$-Dyck path.
Recall that a peak corresponds to the pattern $UD$ in $\mu$.
In the language of a lattice path form $(0,0)$ to $((k+1)n,0)$, a peak corresponds to 
the consecutive steps $(1,k)$ and $(1,-1)$.
We have the weight two on a peak in $k$-Dyck path. This corresponds to introduction of two local paths
with weight one.
The first one is a path $\mu$ with a peak $UD$ and the second one is a path obtained from 
$\mu$ by replacing $UD$ by a step $(2,k-1)$.
Note that the newly obtained paths stay in the first quadrant.
Thus, the statement in (3) holds.
\end{proof}

\begin{remark}
\label{remark:SIII}
Recall that $S^{III}(n;1,0)$ in (3) of Proposition \ref{prop:SIII} counts 
generalized lattice paths. 
We write each step as $U'=(1,k)$, $H'=(2,k-1)$ and $D'=(1,-1)$.
Then, the minimal prime paths are $U'D'^{k}$ and $H'D'^{k-1}$.
In the case of $k$-Schr\"oder paths of type $B$, 
we have two minimal prime paths $UD^{k}$ and $HD^{k-1}$. 
Therefore, the lattice paths appearing in (3) are equivalent 
to $k$-Schr\"oder paths of type $B$. 
\end{remark}

\section{Open problems}
\label{sec:OP}
In section \ref{sec:Gdist2}, we consider a formal power series 
$\zeta(m,r,s)$ with $m=1$.
For a general value of $m$, the formal power series $\zeta(m,-r,s)$ 
is a positive power series with respect to $m,r$ and $s$.
Further, the number of terms with $r^{n}$ is 
equal to the number of $2$-Dyck paths. 
Thus, $\zeta(m,r,s)$ can be regarded as a refinement of $\zeta(m=1,r,s)$.

\paragraph{\bf Problem 1}
Find a statistics on $2$-Dyck paths whose generating function is equal
to $\zeta(m,r,s)$.

From Eq. (\ref{eq:zetam}), the power series $\zeta(m,r,s)$ satisfies 
the recurrence equation of order four.
The generating function of $2$-Dyck paths has degree three.
Further, as already shown in Lemma \ref{lemma:Gcdis2}, the generating function
$G_{n}(m,r,s)$ is the sum of two generating functions, one of 
which satisfies a recurrence equation.
These observation implies that $\zeta(m,r,s)$ should be expressed 
by use of at least two generating functions.
These generating functions may have degree two or three, and 
satisfy a combined recurrence equations.

We expand $\zeta(m,r,s)$ as $\zeta(m,r,s)=\sum_{n\ge 0}\zeta_{n}(m,s)(-rs)^{n}$.
Then, the first few polynomials $\zeta_{n}:=\zeta_{n}(m,s)$ with $n\ge0$ are given by
\begin{align*}
&\zeta_{0}=1, \qquad \zeta_{1}=m, \qquad \zeta_{2}=m^{2}+m(s+s^2), \\
&\zeta_{3}=m^{3}+m^2(2s+3s^2+s^4)+m(2s^3+s^4+s^5+s^6), \\
&\zeta_{4}=m^{4}+m^{3}(3s+5s^2+s^3+2s^4+s^{6})+m^{2}(s^2+6s^3+6s^4+7s^5+2s^6+3s^7+3s^8+s^{10}) \\
&\quad+m(4s^6+2s^7+s^8+3s^9+s^{10}+s^{11}+s^{12}).
\end{align*}
Recall that we have a bijection between a $2$-Dyck path $\mu$ and a ternary tree $T$.
Below, we identify $\mu$ with $T$.
A node in a ternary tree has three edges below it.
Given a ternary tree $T$, we denote by $N(T)$ one plus the number of internal edges 
which are the right-most edge of a parent node. 
Then, $N(T)$ takes the value in $[1,n]$. 
Let us define a polynomial $\zeta'(m,s)$ by 
\begin{align*}
\zeta'_{n}(m,s):=\sum_{T\in\mathtt{Tree}(n;k)}m^{N(T)}s^{|Y(\mu_{0}/\mu)|},
\end{align*}
where we have identified $\mu$ with $T$.
Then, we have $\zeta_{n}(m,s)=\zeta'_{n}(m,s)$ for $1\le n\le 3$. 
However, we have $\zeta_{n}(m,s)\neq\zeta'_{n}(m,s)$ for $n\ge4$.
The reason for this discrepancy is that $\zeta'_{n}(m,s)$ satisfies 
the recurrence equation of degree three, not four.
To obtain a correct solution for Eq. (\ref{eq:zetam}), we need to define 
a new statistics on a tree or a $2$-Dyck path, which coincides with 
$N(T)$ for $n\le3$. 

\paragraph{\bf Problem 2}
Generalize the solution of Problem 1 to the case of $k$-Dyck paths with $k\ge3$.

A natural generalization of $\zeta(m,r,s)$ is to consider the 
generating function of dimer configurations such that the distance of two 
edges with the same color is more than $k$.

\bibliographystyle{amsplainhyper} 
\bibliography{biblio}

\end{document}